\documentclass[12pt,oneside]{amsart}

\usepackage{geometry}                
\usepackage{graphicx}
\usepackage{amssymb}
\usepackage{epstopdf}
\usepackage{color}
\usepackage{bbm, dsfont,hyperref}
\usepackage{tikz}
\usetikzlibrary{3d}

\numberwithin{equation}{section}

\usepackage{verbatim}

\usepackage{amssymb,amsthm,amsmath,amssymb,wrapfig,dsfont}
\DeclareFontFamily{OML}{rsfs}{\skewchar\font'177}
\DeclareFontShape{OML}{rsfs}{m}{n}{ <5> <6> rsfs5 <7> <8> <9>
	rsfs7 <10> <10.95> <12> <14.4> <17.28> <20.74> <24.88> rsfs10 }{}
\DeclareMathAlphabet{\mathfs}{OML}{rsfs}{m}{n}

\usepackage{ulem}
\newtheorem{theorem}{Theorem}
\newtheorem{lemma}{Lemma}[section]

\newtheorem{proposition}[lemma]{Proposition}
\newtheorem{conjecture}{Conjecture}
\theoremstyle{definition}
\newtheorem{definition}{Definition}

\newtheorem{remark}[lemma]{Remark}

\theoremstyle{example}

\theoremstyle{openproblem}

\numberwithin{equation}{section}

\usepackage{amsmath}
\usepackage{mathrsfs}

\usepackage{verbatim}
\usepackage{amsmath,amssymb,amsthm}             
\usepackage{amssymb}             
\usepackage{amsfonts}            
\usepackage{mathrsfs}            
\usepackage{amsthm}
\usepackage{algorithm}  
\usepackage{algorithmicx}  
\usepackage{algpseudocode}  
\usepackage{mathtools}
\usepackage{commath}
\usepackage{physics}
\usepackage{bm}
\usepackage{graphicx}

\usepackage{float}
\usepackage{listings}
\usepackage{subfigure}
\usepackage{multirow}
\usepackage{diagbox}
\usepackage{color}
\usepackage{bbm}
\usepackage[export]{adjustbox}
\usepackage{enumerate}
\usepackage{bbm}

\usepackage{amsmath}
\usepackage{mathrsfs}

\newcommand{\RNum}[1]{\uppercase\expandafter{\romannumeral #1\relax}}

\newcommand{\specificthanks}[1]{\@fnsymbol{#1}}

\DeclareFontFamily{U}{mathx}{}
\DeclareFontShape{U}{mathx}{m}{n}{<-> mathx10}{}
\DeclareSymbolFont{mathx}{U}{mathx}{m}{n}
\DeclareMathAccent{\widehat}{0}{mathx}{"70}
\DeclareMathAccent{\widecheck}{0}{mathx}{"71}

\begin{document}

	\title{Vertex-removal stability and the least positive value of harmonic measures}
	
	\author{Zhenhao Cai$^1$}
	\address[Zhenhao Cai]{School of Mathematical Sciences, Peking University, Beijing, China}
	\email{caizhenhao@pku.edu.cn}
	\thanks{$^1$School of Mathematical Sciences, Peking University, Beijing, China}

	\author{Gady Kozma$^2$}
	\address[Gady Kozma]{Faculty of Mathematics and Computer Science, Weizmann Institute of Science, Rehovot, Israel}
	\email{gady.kozma@weizmann.ac.il}
	\thanks{$^2$Faculty of Mathematics and Computer Science, Weizmann Institute of Science, Rehovot, Israel}

	\author{Eviatar B. Procaccia$^3$}
	\address[Eviatar B. Procaccia]{Faculty of Data and Decision Sciences, Technion - Israel Institute of Technology, Haifa, Israel}
	\email{eviatarp@technion.ac.il}
	\thanks{$^3$Faculty of Data and Decision Sciences, Technion - Israel Institute of Technology, Haifa, Israel}

	\author{Yuan Zhang$^4$}
	\address[Yuan Zhang]{Center for Applied Statistics and School of Statistics, Renmin University of China, Beijing, China}
	\email{zhang\_probab@ruc.edu.cn}
	\thanks{$^4$Center for Applied Statistics and School of Statistics, Renmin University of China, Beijing, China}

	\maketitle
	
	\tableofcontents
	
	\begin{abstract}
		
		We prove that for $\mathbb{Z}^d$ ($d\ge 2$), the vertex-removal stability of harmonic measures (i.e. it is feasible to remove some vertex while changing the harmonic measure by a bounded factor) holds if and only if $d=2$. The proof mainly relies on geometric arguments, with a surprising use of the discrete Klein bottle. Moreover, a direct application of this stability verifies a conjecture of Calvert, Ganguly and Hammond \cite{calvert2021collapse} for the exponential decay of the least positive value of harmonic measures on $\mathbb{Z}^2$. Furthermore, the analogue of this conjecture for $\mathbb{Z}^d$ with $d\ge 3$ is also proved in this paper, despite vertex-removal stability no longer holding. 
		
	\end{abstract}

	\section{Introduction}	
	In this paper, we study the (discrete) harmonic measure (from infinity) on graphs, with a particular focus on the $d$-dimensional integer lattices $\mathbb{Z}^d$ (we assume $d\ge 2$ throughout this paper). Roughly speaking, the harmonic measure of a finite set is the hitting distribution of this set by a random walk on the graph starting from infinity. To facilitate understanding, we first review the definition of harmonic measure on $\mathbb{Z}^d$ (a more general definition will be provided later in Section \ref{sebsection_Q2}). A random walk on $\mathbb{Z}^d$ starting from $x\in \mathbb{Z}^d $ is a discrete-time Markov process $\{S_n\}_{n\ge 0}$ (we denote its law by $\mathbb{P}_x$) such that $\mathbb{P}_x\left( S_0=x \right)=1$ and 
	\begin{equation}
		\mathbb{P}_x\left( S_{n+1}=z \mid S_n=y \right)=  (2d)^{-1} \cdot \mathbbm{1}_{y\sim z},\ \forall y,z\in \mathbb{Z}^d\ \text{and}\ n\ge 0,
	\end{equation}
	where ``$y\sim z$'' means that $y$ and $z$ are adjacent (i.e. $|y-z|=1$, where $|\cdot|$ is the Euclidean norm). For any non-empty, finite $A\subset \mathbb{Z}^d$ and $y\in A$, the harmonic measure of $A$ at $y$ is defined as 
	\begin{equation}\label{1.1}
		\mathbb{H}_A(y):=\lim\limits_{|x|\to \infty} \mathbb{P}_x(S_{\tau_A}=y\mid \tau_A<\infty),
	\end{equation}
	where $\tau_A:=\inf\{n\ge 0: S_n\in A\}$ (with $\inf \emptyset=\infty$ for completeness) is the first time when $\{S_n\}_{n\ge 0}$ hits $A$. Refer to \cite[Theorem 2.1.3]{lawler2013intersections} for the existence of the limit in (\ref{1.1}). It follows from the definition that $\mathbb{H}_A(\cdot)$ is a probability measure on $A$ (i.e. $\sum_{y\in A}\mathbb{H}_A(y)=1$). Notably, the continuous analogue of $\mathbb{H}_A(\cdot)$ can be obtained by replacing $\mathbb{Z}^d$ and random walk with $\mathrm{R}^d$ and Brownian motion respectively, and is highly related to the Dirichlet problem in the partial differential equation (PDE) field (see \cite{kesten1991relations} and \cite[Sections 3 and 8]{morters2010brownian}). It is also worth mentioning that there are numerous widely-studied models in statistical physics based on harmonic measures, such as Diffusion Limited Aggregation (DLA) \cite{kesten1987long,kesten1990upper,witten1981diffusion,witten1983diffusion}, Stationary DLA \cite{ mu2022scaling,procaccia2020stationary}, Dielectric Breakdown Model (DBM) \cite{losev2023long,niemeyer1984fractal} and Harmonic Activation and Transport (HAT) \cite{calvert2021existence,calvert2021collapse} for the discrete harmonic measure, as well as 	Hastings-Levitov \cite{berger2022growth, norris2012hastings, silvestri2017fluctuation}, Anisotropic Hastings–Levitov \cite{johansson2012scaling} and Aggregate Loewner Evolution (ALE) \cite{sola2019one} for the continuous one. Readers may refer to \cite{lawler1994random,olegovic2021diffusion} for excellent accounts.

	This study mainly focuses on the following two fundamental questions.
	
	\begin{itemize}
		\item[\textbf{Q1}:]  \textit{How does removing a vertex from a set affect the harmonic measure of another vertex?}

		\item[\textbf{Q2}:]   \textit{What is the least positive value of the harmonic measure that a vertex can have in a set of fixed cardinality? }
	\end{itemize}
	
	In fact, as we will show later, these two questions are highly related to one another, and correspond to the two main components of this paper respectively.

	\subsection{Vertex-removal stability of the harmonic measure}\label{sebsection_Q1} Here we formulate \textbf{Q1}. In fact, for any $A\subset \mathbb{Z}^d$, $A'\subset A$ and $y\in A'$, we have $\mathbb{H}_{A}(y)\le \mathbb{H}_{A'}(y)$ since $\{S_{\tau_A}=y \}\subset \{S_{\tau_{A'}}=y \}$ and $\{\tau_{A}<\infty\} \supset \{\tau_{A'}<\infty\}$. Especially, removing a vertex $z\in A\setminus \{y\}$ from $A$ leads to an increase of the harmonic measure $\mathbb{H}_A(y)$. In light of this, for any $A\subset \mathbb{Z}^d$, $y\in A$ with $\mathbb{H}_A(y)>0$, and any $z\in A\setminus \{y\}$, we define the price of removing $z$ with respect to $\mathbb{H}_A(y)$ as 
	\begin{equation}\label{def_rho}
		\rho_{A,y}(z):=\frac{\mathbb{H}_{A\setminus \{z\}}(y)}{\mathbb{H}_A(y)}\in \left[1,\infty \right),
	\end{equation}
	where we require that $\mathbb{H}_A(y)>0$ to avoid the case of $0$ being the divisor. However, $\rho_{A,y}(z)$ can be arbitrarily large. To see this, readers may refer to the following ``tube'' example in Figure \ref{fig:tube}, where removing the point $z$ enlarges the harmonic measure at $y$ from exponentially small (with respect to $|A|$) to polynomially small (since after removing $z$, the random walk may directly reach $y$ from the right side; in addition, by adapting the proof of \cite[Equation (2.41)]{lawler2013intersections}, it can be shown that $\mathbb{H}_{A}(z)$ decays polynomially with respect to $|A|$).

	  
	  \begin{figure}[h!]
	  	\begin{tikzpicture}[scale=0.7]
	  		\draw[step=1cm,blue,thin, dotted] (-6,-1) grid (6,3);
	  		\node at (-1,0) [circle,fill=black,inner sep=2]{}; \node at (-2,0) [circle,fill=black,inner sep=2pt]{};\node at (-3,0) [circle,fill=black,inner sep=2pt]{};\node at (-4,0) [circle,fill=black,inner sep=2pt]{}; \node at (-5,0) [circle,fill=black,inner sep=2pt]{};
	  		\node at (-1,2) [circle,fill=black,inner sep=2pt]{}; \node at (-2,2) [circle,fill=black,inner sep=2pt]{};\node at (-3,2) [circle,fill=black,inner sep=2pt]{};\node at (-4,2) [circle,fill=black,inner sep=2pt]{}; \node at (-5,2) [circle,fill=black,inner sep=2pt]{};
	  		\node at (0,0) [circle,fill=black,inner sep=2pt]{}; \node at (4,1) [circle,fill=red,inner sep=2pt]{};  \node at (3.5,1) []{$y$};  \node at (5.5,1) []{$z$}; 
	  		\node at (1,0) [circle,fill=black,inner sep=2pt]{};	\node at (2, 0)[circle,fill=black,inner sep=2pt]{} ;\node at (3,0) [circle,fill=black,inner sep=2pt]{};\node at (4,0) [circle,fill=black,inner sep=2pt]{};\node at (5,0) [circle,fill=black,inner sep=2pt]{};\node at (5,1) [circle,fill=black,inner sep=2pt]{};\node at (5,2) [circle,fill=black,inner sep=2pt]{};\node at (4,2) [circle,fill=black,inner sep=2pt]{};\node at (3,2) [circle,fill=black,inner sep=2pt]{};\node at (2,2) [circle,fill=black,inner sep=2pt]{};\node at (1,2) [circle,fill=black,inner sep=2pt]{};\node at (0,2) [circle,fill=black,inner sep=2pt]{};\end{tikzpicture}
	  	\caption{ The ``tube'' example for unbounded $\rho_{A,y}(z)$ \label{fig:tube}}
	  \end{figure}
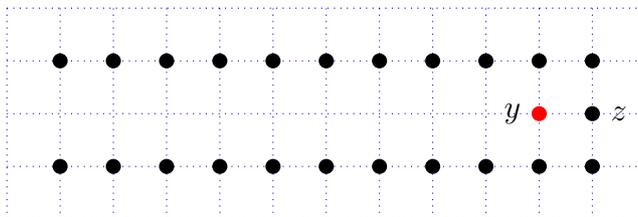
	  

	When discussing the vertex-removal stability, we aim to study whether for all $A\subset \mathbb{Z}^d$ and $y\in A$ with $\mathbb{H}_A(y)>0$ and $|A|\ge 2$ (where $|A|$ is the cardinality of the set $A$), there exists a vertex $z\in A\setminus \{y\}$ such that the price of removing $z$ with respect to $\mathbb{H}_A(y)$ is uniformly upper-bounded, which motivates us to present the following definition:
		\begin{definition}\label{def1}
		We define the vertex-removal stability constant for harmonic measure on $\mathbb{Z}^d$ as
		\begin{equation}
			\psi(\mathbb{Z}^d):= \sup_{A\subset \mathbb{Z}^d,y\in A:\mathbb{H}_A(y)>0,|A|\ge 2} \min_{z\in A\setminus \{y\}}  \rho_{A,y}(z).
		\end{equation}
	\end{definition}
	Our main result is the following:
\begin{theorem}\label{thm1}
	$\psi(\mathbb{Z}^d)<\infty$ if and only if $d=2$. 
\end{theorem}
	
	Among the proof of Theorem \ref{thm1}, in addition to the probabilistic properties of random walks, geometric arguments play a pivotal role. Specifically, to prove $\psi(\mathbb{Z}^2)<\infty$, we introduce a special type of vertices named ``marginal vertices'' (which are candidates for removal with a bounded price) and then establish its existence in the case of $\mathbb{Z}^2$. Moreover, the proof of $\psi(\mathbb{Z}^d)=\infty$ for $d\ge 3$ surprisingly involves the discrete version of a classic geometric object, the Klein bottle. Furthermore, we observe that both the existence of marginal vertices and the construction of discrete Klein bottles can be naturally extended to a large class of graphs, which inspires us to consider a generalization of Theorem \ref{thm1} in Remarks \ref{generalize_2d} and \ref{generalize_3d}.

%


	\subsection{Least positive value of harmonic measures} \label{sebsection_Q2}

	Before turning to \textbf{Q2}, let us review a more general definition of harmonic measure called ``wired harmonic measure (from infinity)'', which is introduced in \cite[Exercise 2.50 (c)]{lyons2017probability}. For any simple (i.e. there is no more than one edge between each two vertices, and no edge that connects a vertex to itself), locally finite (i.e. every vertex has a finite degree) and connected graph $\mathfs{G}=(\mathfs{V},\mathfs{E})$, the random walk on $\mathfs{G}$ starting from $x\in \mathfs{V}$ is a discrete-time Markov process $\{S_n\}_{n\ge 0}$ (we also denote its law by $\mathbb{P}_x$) such that $\mathbb{P}_x(S_0=x)=1$ and 
	\begin{equation}
		\mathbb{P}_x(S_{n+1}=z \mid S_n=y)=  [\mathrm{deg}(y)]^{-1}\cdot\mathbbm{1}_{y\sim z},\ \forall y,z\in \mathfs{V}\ \text{and}\ n\ge 0,
	\end{equation} 
	where ``$y\sim z$'' is equivalent to ``$\{y,z\}\in \mathfs{E}$'', and $\mathrm{deg}(y):= |\{v\in \mathfs{V}: v\sim y \}|$. Assume that in $\mathfs{V}$ there is a prefixed vertex, which we denote by $v_*$. For each $v\in \mathfs{V}$, let $\|v\|$ be the graph distance between $v_*$ and $v$ (i.e. the minimum length of a path on $\mathfs{G}$ from $v_*$ to $v$). For any $r\ge 0$, we denote $\mathbf{B}(r):= \{v\in \mathfs{V}: \|v\|\le r \}$. Let $\mathfs{G}_r$ be the graph obtained from $\mathfs{G}$ by identifying all vertices in $\mathfs{V}\setminus \mathbf{B}(r)$ to a single vertex $\zeta_r$ and then removing all edges connecting $\zeta_r$ to itself. For any non-empty, finite $A\subset \mathfs{V}$ and $y\in A$, the wired harmonic measure of $A$ at $y$ is defined as 
	\begin{equation}\label{def_harmonic_measure_2}
		\mathbb{H}_A(y):=\lim_{r\to \infty}\mathbb{P}_{\zeta_r}(S_{\tau_A}=y\mid \tau_A<\tau^+_{\{\zeta_r\}}),
	\end{equation}
	where the law on the RHS is the probability distribution for the random walk on $\mathfs{G}_r$ starting from $\zeta_r$, and $\tau^+_{F}:=\inf\{n\ge 1: S_n\in F\}$. However, the existence of the limit in (\ref{def_harmonic_measure_2}) is not valid for all graphs (only confirmed for transient graphs; see \cite[Exercise 2.50 (c)]{lyons2017probability}), and is an interesting topic on its own right (although it is not the focus of this study). Readers may refer to \cite{boivin2013existence} and \cite[Section 10.7]{lyons2017probability} for related results. Here we employ the definition in (\ref{def_harmonic_measure_2}) since it is consistent with the one presented in (\ref{1.1}), and is valid for all graphs mentioned in this paper.

	Extremal values of harmonic measures have been extensively studied in mathematics and physics, mostly in the context of interacting particle systems driven by diffusion (such as DLA, DBM, Hastings-Levitov and ALE mentioned above). Particularly, in the case of $\mathbb{Z}^d$, Kesten \cite{kesten1987hitting} proved the following polynomial upper bounds (this type of estimate is commonly named ``discrete Beurling's estimate'') for harmonic measures: for any $d\ge 2$, there exists a constant $C(d)>0$ such that for every non-empty, finite, connected set $A\subset \mathbb{Z}^d$,
	\begin{equation}\label{estimate_Kesten}
		\max_{y\in A}\mathbb{H}_A(y)\le \left\{\begin{array}{ll}C(2)[\mathrm{rad}(A)]^{-\frac{1}{2}} & \text{for } d=2\\
			~&~\\
			C(d) |A|^{\frac{2-d}{d}}& \text{for } d\ge3\end{array}\right.,
	\end{equation}
	where $\mathrm{rad}(A):=\max_{y\in A}|y|$. Note that for $d=2$, while the endpoints of line segments achieve the RHS of \eqref{estimate_Kesten} up to a multiplicative constant, it remains uncertain whether these points constitute the maximum in \eqref{estimate_Kesten}. With the help of (\ref{estimate_Kesten}), Kesten \cite{kesten1987hitting} derived an upper bound for the growth rate of DLA. Readers may refer to \cite[Sections 2.5 and 2.6]{lawler2013intersections} for other versions of discrete Beurling's estimates and more applications on DLA. Subsequently, Lawler and Limic \cite{lawler2004beurling}, and Benjamini and Yadin \cite{benjamini2017upper} extended the discrete Beurling's estimates to more general random walks. Furthermore, Makarov \cite{makarov1985distortion} established that for any domain $\Omega \subset \mathbb{R}^2$ whose boundary is a Jordan curve, the Hausdorff dimension of the support of its continuous harmonic measure is $1$. Later, Lawler \cite{lawler1993discrete} proved its discrete analogue as follows: on $\mathbb{Z}^2$, for any $\frac{1}{2}<\alpha<1$ and $\beta<\alpha-\frac{1}{2}$, there exists a constant $C(\alpha,\beta)>0$ such that for every connected $A\subset\mathbb{Z}^2$ with $\mathrm{rad}(A)=R$,
	\begin{equation}
			\mathbb{H}_A\Big(\big\{y\in A: R^{-1}e^{-(\log R)^\alpha}\le \mathbb{H}_A(y)\le R^{-1}e^{(\log R)^\alpha}\big\}\Big)\ge 1-\frac{C}{(\log R)^\beta}.
	\end{equation}
	Moreover, Benjamini \cite{benjamin1997support} established that the cardinality of the support of the harmonic measure of any set contained in a box on $\mathbb{Z}^d$ with side length $R$ is at most $o(R^d)$, and afterward, a quantitative version was given by Bolthausen and M{\"u}nch-Berndl \cite{bolthausen2001quantitative}. As shown above, these results concentrated on either the upper bounds, or the typical values, or the supports of harmonic measures, and only considered the sets with connectivity or bounded radius. Our understanding of the minimal value of harmonic measures of general, possibly disconnected or spread-out sets, remains limited. Recently, in the study of HAT \cite{calvert2021existence,calvert2021collapse}, the least positive value of harmonic measures of general sets (which is the focus of \textbf{Q2}) played an important role in bounding the time of collapse in HAT.

	Now let us formulate \textbf{Q2}. Assume that the harmonic measure on $\mathfs{G}=(\mathfs{V},\mathfs{E})$ is well-defined. For any integer $n\ge 2$,	we denote the least positive value that a vertex can have in a set
	of cardinality $n$ by 
	\begin{equation}\label{def_Mn}
		\mathfs{M}_n(\mathfs{G})=\inf_{A\subset \mathfs{V},y\in A:|A|=n,\mathbb{H}_A(y)>0} \mathbb{H}_A(y). 
	\end{equation}


	In fact, the behavior of $\mathfs{M}_n(\mathfs{G})$ can vary significantly over different types of graphs. For example, consider the recurrent graph in Figure \ref{fig:most_recurrent}, which is obtained from $\mathbb{Z}$ by attaching an additional edge to each vertex. Denote $A^{(k)}:=\{(0,0),(-1,0), (k,1)\}$ for integer $k\ge 1$. Notice that for any $j\ge k$, the number of times when a random walk starting from $(j,0)$ visits $(k,0)$ before hitting $(0,0)$ is a geometry random variable with average $2k$. Moreover, in each visit to $(k,0)$ there is a probability of $1/3$ to hit the vertex $(k,1)\in A^{(k)}$. As a result, we have $\lim_{k\to\infty}\mathbb{H}_{A^{(k)}}((0,0))= 0$ and hence, $\mathfs{M}_3(\mathfs{G})=0$. This means that one can get an arbitrarily small positive value of the harmonic measure by a set containing only three vertices! 
	
	\begin{figure}[h!]
		\begin{tikzpicture}[scale=0.8]
			\node at (-1,0) [circle,fill=blue,inner sep=1pt]{}; \node at (-2,0) [circle,fill=blue,inner sep=1pt]{};\node at (-3,0) [circle,fill=blue,inner sep=1pt]{};\node at (-4,0) [circle,fill=blue,inner sep=1pt]{}; \node at (-5,0) [circle,fill=red,inner sep=2pt]{};\node at (-6,0) [circle,fill=black,inner sep=2pt]{};\node at (-6,1) [circle,fill=blue,inner sep=1pt]{};
			
			\node at (-5,-0.5) [] { $(0,0)$};
			
			\node at (-7,0.3) [] { $(-1,0)$};
			
			\node at (-1,1) [circle,fill=blue,inner sep=1pt]{}; \node at (-2,1) [circle,fill=blue,inner sep=1pt]{};\node at (-3,1) [circle,fill=blue,inner sep=1pt]{};\node at (-4,1) [circle,fill=blue,inner sep=1pt]{}; \node at (-5,1) [circle,fill=blue,inner sep=1pt]{};
			\node at (0,0) [circle,fill=blue,inner sep=1pt]{}; \node at (4,1) [circle,fill=red,inner sep=1pt]{}; 
			\node at (1,0) [circle,fill=blue,inner sep=1pt]{};	\node at (2, 0)[circle,fill=blue,inner sep=1pt]{} ;\node at (3,0) [circle,fill=blue,inner sep=1pt]{};\node at (4,0) [circle,fill=blue,inner sep=1pt]{};\node at (5,0) [circle,fill=blue,inner sep=1pt]{};\node at (5,1) [circle,fill=blue,inner sep=1pt]{};\node at (4,1) [circle,fill=black,inner sep=2pt]{};
			\node at (4,1.5) []{$(k,1)$};
			
			\node at (3,1) [circle,fill=blue,inner sep=1pt]{};\node at (2,1) [circle,fill=blue,inner sep=1pt]{};\node at (1,1) [circle,fill=blue,inner sep=1pt]{};\node at (0,1) [circle,fill=blue,inner sep=1pt]{};
			
			\draw[blue, dashed] (-7,0)--(6,0);	
			\foreach \x in {-6,...,5}
			{
				\draw[blue, dashed] (\x,0)--(\x,1);	
			}
			
		\end{tikzpicture}
		\caption{The ``$\mathbb{Z}$ with hairs'' example such that $\mathfs{M}_3(\mathfs{G})=0$} 
		\label{fig:most_recurrent}
	\end{figure}
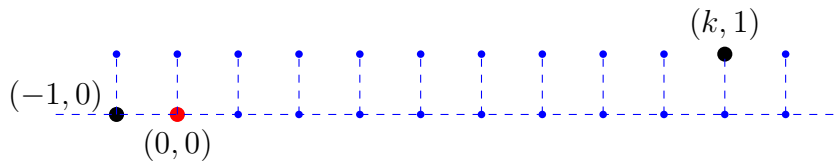

	For transient graphs, putting a faraway obstacle will no longer have a significant impact on the harmonic measure, as a remote vertex is almost never reached by a random walk. Consequently, unlike the example in Figure \ref{fig:most_recurrent} where the set tends to become sparse to achieve a smaller harmonic measure, the optimal set that minimizes the harmonic measure is more likely to take on a ``tunnel'' shape. For instance, for the 3-regular tree $\mathbb{T}_3$, it is not difficult to see that the least positive value of the escape probability $\mathrm{Es}_A(y):=\mathbb{P}_{y}\left(\tau_{A}^+= \infty\right)$ (where $|A|=n$ and $y\in A$) is achieved by the example in Figure \ref{fig:3tree}. Combined with the relation between the harmonic measure and the escape probability (see (\ref{2.14})), it implies that $\mathfs{M}_n(\mathbb{T}_3)$ decays exponentially with respect to $n$, and its decay rate is realized by the example in Figure \ref{fig:3tree}. More precisely, we have $\mathfs{M}_n(\mathbb{T}_3)= (\frac{3-\sqrt{5}}{2})^{n+o(n)}$. To see this, consider the number sequence $\{q_i\}_{i=1}^{n}$, representing the escape probability through a tunnel of length $n$, when starting $i$ steps from the tunnel's end,  with $q_1=c\in (0,1)$, $q_n=0$ and $q_i=\frac{1}{3}(q_{i-1}+q_{i+1})$ for $2\le i\le n-1$.


	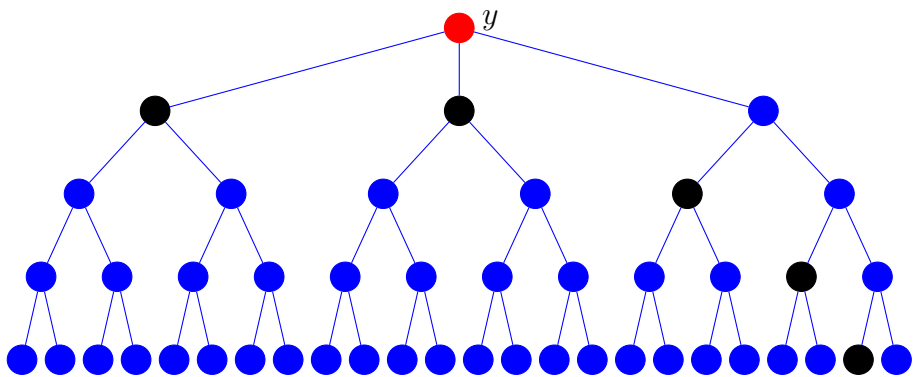
\begin{figure}[h!]
		\begin{tikzpicture}
			[
			level distance=1.1cm,
			level 1/.style={sibling distance=4cm},
			level 2/.style={sibling distance=2cm},
			level 3/.style={sibling distance=1cm},
			level 4/.style={sibling distance=0.5cm},
			main node/.style={circle,blue, fill, draw, minimum size=0.1cm},
			every edge/.style={blue, draw},
			]
			\node at (0.41,0.1) {$y$};
			\node[main node,red] {}
			child[-, blue] {node[main node,black] {}
				child {node[main node] {}
					child {node[main node] {}
						child {node[main node] {}}
						child {node[main node] {}}
					}
					child[-, blue] {node[main node] {}
						child {node[main node] {}}
						child {node[main node] {}}
					}
				}
				child[-, blue] {node[main node] {}
					child {node[main node] {}
						child {node[main node] {}}
						child {node[main node] {}}
					}
					child[-, blue] {node[main node] {}
						child {node[main node] {}}
						child {node[main node] {}}
					}
				}
			}
			child[-, blue] {node[main node,black] {}
				child {node[main node] {}
					child {node[main node] {}
						child {node[main node] {}}
						child {node[main node] {}}
					}
					child {node[main node] {}
						child {node[main node] {}}
						child {node[main node] {}}
					}
				}
				child {node[main node] {}
					child {node[main node] {}
						child {node[main node] {}}
						child {node[main node] {}}
					}
					child {node[main node] {}
						child {node[main node] {}}
						child {node[main node] {}}
					}
				}
			}
			child[-, blue] {node[main node] {}
				child {node[main node,black] {}
					child {node[main node] {}
						child {node[main node] {}}
						child {node[main node] {}}
					}
					child {node[main node] {}
						child {node[main node] {}}
						child {node[main node] {}}
					}
				}
				child[-, blue] {node[main node] {}
					child {node[main node,black]{}
						child {node[main node]{}}
						child {node[main node]{}}
					}
					child {node[main node]{}
						child {node[main node,black] {}}
						child {node[main node] {}}
					}
				}
			};
		\end{tikzpicture}
		\caption{The ``tunnel'' example reaching the decay rate of $\mathfs{M}_n(\mathbb{T}_3)$ \label{fig:3tree}}
	\end{figure}

	In the case of $\mathbb{Z}^2$, for any $A\subset \mathbb{Z}^2$ and $y\in A$ with $\mathbb{H}_A(y)>0$, it follows from Definition \ref{def1} that either $\mathbb{H}_A(y)=1$ (i.e. $A=\{y\}$), or there exists $z\in A\setminus \{y\}$ such that $\rho_{A,y}(z)\le \psi(\mathbb{Z}^2)$ (which is finite by Theorem \ref{thm1}). Thus, by induction, 
	\begin{equation}\label{lowerbound_psi}
	\ \ \ \ \ \ \ \ \ \ 	\mathbb{H}_A(y)\ge [\psi(\mathbb{Z}^2)]^{-(|A|-1)}, \ \ \forall A\subset \mathbb{Z}^2,y\in A\ \text{with}\ \mathbb{H}_A(y)>0.
	\end{equation}
	Moreover, $\mathfs{M}_n(\mathbb{Z}^2)$ is bounded from above by the harmonic measure at $y$ of the example presented in Figure \ref{fig:tube}, which decays exponentially. To sum up, we obtain the following estimates for $\mathfs{M}_n(\mathbb{Z}^2)$: 
		\begin{theorem}\label{coro1}
		There exist constants $C_1,c_1>0$ such that for any $n\ge 2$, 
		\begin{equation}\label{1.4}
		 e^{-C_1n}	\le \mathfs{M}_n(\mathbb{Z}^2) \le e^{-c_1n}. 
		\end{equation}
	\end{theorem}
	As mentioned above, the estimation of $\mathfs{M}_n(\mathbb{Z}^2)$ has emerged in the study of HAT. Specifically, Calvert, Ganguly and Hammond \cite[Theorem 1.9]{calvert2021collapse} proved that $\mathfs{M}_n(\mathbb{Z}^2) \ge e^{-Cn\log(n)}$; in addition, if restricting to connected sets, then this lower bound can be improved to $e^{-Cn}$. In \cite{calvert2021collapse}, they further conjectured that the exponentially decaying lower bound can be proven without the connectivity condition. This conjecture is now settled by Theorem \ref{coro1}.

	Given Theorem \ref{coro1}, a natural subsequent question is which set achieves (or closely approaches) $\mathfs{M}_n(\mathbb{Z}^2)$. Notably, Lawler \cite[Section 2.5]{lawler2013intersections} showed that adding a faraway vertex will approximately halve the harmonic measure. Inspired by Lawler's observation, we consider a very sparse set in $\mathbb{Z}^2$ for which $(0,0)$ has exponentially small harmonic measure: $A_n:=\{(0,0)\}\cup \{ ({^{k}2},0): 0\le k\le n-2 \}$ (see Figure \ref{fig:2^n}), where ${^{k}2}$ is the $k$-th tetration of $2$ (i.e. ${^{0}2}:=1$ and ${^{k}2}:=2^{[{^{(k-1)}2}]}$ for $k\ge 1$). It can be proved by induction that 
	\begin{equation}
		\mathbb{H}_{A_n}((0,0))=\prod_{k=1}^{n-1}\bigg[  \frac{1}{2}-O\Big( \frac{{^{(k-1)}2}}{{^{k}2}}\Big) \bigg].
	\end{equation}
	\begin{figure}[H]
		\hskip-1.5cm
		\begin{tikzpicture}[scale=0.5]
			\draw[step=1cm,blue,ultra thin] (-20,0) grid (-1,2);
			\node at (-19,-0.01) []{$(0,0)$};  
			\node at (-19,1) [circle,fill=red,inner sep=2pt]{}; 	
			\node at (-18,1) [circle,fill=black,inner sep=2pt]{}; 
			\node at (-17,1) [circle,fill=black,inner sep=2pt]{};
			\node at (-15,1) [circle,fill=black,inner sep=2pt]{};
			\node at (-3,1) [circle,fill=black,inner sep=2pt]{};
		\end{tikzpicture}
		\caption{ A sparse set with exponentially small harmonic measure \label{fig:2^n}}
	\end{figure}
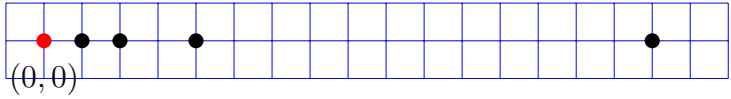
	However, this sparse set doesn't give the fastest exponential rate, as can be seen by the following example. Calvert, Ganguly and Hammond \cite[Example 1.11]{calvert2021collapse} proposed an increasing sequence of sets $\{D_n\}_{n\ge 2}$ such that $|D_n|=n$ and $\mathbb{H}_{D_n}((0,0))=(2+\sqrt{3})^{-2n+o(n)}$ (where the base number $(2+\sqrt{3})^{-1}$ is derived from the number sequence $\{q_i\}_{i=1}^{2n}$ with $q_{i}=\frac{1}{4}(q_{i-1}+q_{i+1})$, representing the probability to walk along a specified path of length $i$, and the power $2n+o(n)$ is the length of the tunnel). See Figure \ref{fig:spiral} for the construction of $D_n$. Furthermore, they \cite[Conjecture 1.10]{calvert2021collapse} conjectured that $\{D_n\}_{n\ge 2}$ approximately realizes the least positive value of harmonic measures on $\mathbb{Z}^2$. I.e., 
		\begin{equation}\label{1.5}
			\lim\limits_{n\to \infty} -\frac{1}{n}\log(\mathfs{M}_n(\mathbb{Z}^2)) = 2\log(2+\sqrt{3}). 
		\end{equation} 
	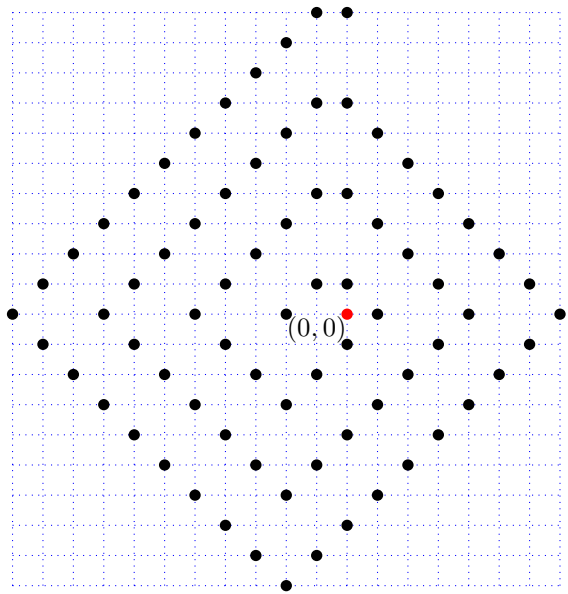
\begin{figure}[h!]
		\begin{tikzpicture}[scale=0.4]
			\draw[step=1cm,blue,thin, dotted] (-11,-9) grid (7,10);
			\node at (0,0) [circle,fill=red,inner sep=1.5pt]{}; \node at (-1,-0.5) []{\footnotesize $(0,0)$};
			\node at (0,1) [circle,fill=black,inner sep=1.5pt]{};       \node at (0, -1)[circle,fill=black,inner sep=1.5pt]{} ;\node at (-1,1) [circle,fill=black,inner sep=1.5pt]{};\node at (-2,0) [circle,fill=black,inner sep=1.5pt]{};\node at (1,0) [circle,fill=black,inner sep=1.5pt]{};\node at (-1,-2) [circle,fill=black,inner sep=1.5pt]{};\node at (-2,-3) [circle,fill=black,inner sep=1.5pt]{};\node at (-3,-2) [circle,fill=black,inner sep=1.5pt]{};\node at (-4,-1) [circle,fill=black,inner sep=1.5pt]{};\node at (-5,0) [circle,fill=black,inner sep=1.5pt]{};\node at (-4,1) [circle,fill=black,inner sep=1.5pt]{};\node at (-3,2) [circle,fill=black,inner sep=1.5pt]{};\node at (-2,3) [circle,fill=black,inner sep=1.5pt]{};\node at (-1,4) [circle,fill=black,inner sep=1.5pt]{};\node at (0,4) [circle,fill=black,inner sep=1.5pt]{};\node at (1,3) [circle,fill=black,inner sep=1.5pt]{};\node at (2,2) [circle,fill=black,inner sep=1.5pt]{};\node at (3,1) [circle,fill=black,inner sep=1.5pt]{};\node at (4,0) [circle,fill=black,inner sep=1.5pt]{};\node at (3,-1) [circle,fill=black,inner sep=1.5pt]{};\node at (2,-2) [circle,fill=black,inner sep=1.5pt]{};\node at (1,-3) [circle,fill=black,inner sep=1.5pt]{};\node at (0,-4) [circle,fill=black,inner sep=1.5pt]{};\node at (-1,-5) [circle,fill=black,inner sep=1.5pt]{};\node at (-2,-6) [circle,fill=black,inner sep=1.5pt]{};\node at (-3,-5) [circle,fill=black,inner sep=1.5pt]{};\node at (-4,-4) [circle,fill=black,inner sep=1.5pt]{};\node at (-5,-3) [circle,fill=black,inner sep=1.5pt]{};\node at (-6,-2) [circle,fill=black,inner sep=1.5pt]{};\node at (-7,-1) [circle,fill=black,inner sep=1.5pt]{};\node at (-8,0) [circle,fill=black,inner sep=1.5pt]{}; \node at (-7,1) [circle,fill=black,inner sep=1.5pt]{};\node at (-6,2) [circle,fill=black,inner sep=1.5pt]{};
			\node at (-5,3) [circle,fill=black,inner sep=1.5pt]{};\node at (-4,4) [circle,fill=black,inner sep=1.5pt]{};\node at (-3,5) [circle,fill=black,inner sep=1.5pt]{};\node at (-2,6) [circle,fill=black,inner sep=1.5pt]{};\node at (-1,7) [circle,fill=black,inner sep=1.5pt]{};\node at (0,7) [circle,fill=black,inner sep=1.5pt]{};\node at (1,6) [circle,fill=black,inner sep=1.5pt]{};\node at (2,5) [circle,fill=black,inner sep=1.5pt]{};\node at (3,4) [circle,fill=black,inner sep=1.5pt]{};\node at (4,3) [circle,fill=black,inner sep=1.5pt]{};\node at (5,2) [circle,fill=black,inner sep=1.5pt]{};\node at (6,1) [circle,fill=black,inner sep=1.5pt]{};\node at (7,0) [circle,fill=black,inner sep=1.5pt]{};\node at (6,-1) [circle,fill=black,inner sep=1.5pt]{};\node at (5,-2) [circle,fill=black,inner sep=1.5pt]{};\node at (4,-3) [circle,fill=black,inner sep=1.5pt]{};\node at (3,-4) [circle,fill=black,inner sep=1.5pt]{};\node at (2,-5) [circle,fill=black,inner sep=1.5pt]{};\node at (1,-6) [circle,fill=black,inner sep=1.5pt]{};\node at (0,-7) [circle,fill=black,inner sep=1.5pt]{};\node at (-1,-8) [circle,fill=black,inner sep=1.5pt]{};\node at (-2,-9) [circle,fill=black,inner sep=1.5pt]{};\node at (-3,-8) [circle,fill=black,inner sep=1.5pt]{};\node at (-4,-7) [circle,fill=black,inner sep=1.5pt]{};
			\node at (-5,-6) [circle,fill=black,inner sep=1.5pt]{};\node at (-6,-5) [circle,fill=black,inner sep=1.5pt]{};\node at (-7,-4) [circle,fill=black,inner sep=1.5pt]{};\node at (-8,-3) [circle,fill=black,inner sep=1.5pt]{};\node at (-9,-2) [circle,fill=black,inner sep=1.5pt]{};\node at (-10,-1) [circle,fill=black,inner sep=1.5pt]{};\node at (-11,0) [circle,fill=black,inner sep=1.5pt]{};\node at (-10,1) [circle,fill=black,inner sep=1.5pt]{};\node at (-9,2) [circle,fill=black,inner sep=1.5pt]{};\node at (-8,3) [circle,fill=black,inner sep=1.5pt]{};\node at (-7,4) [circle,fill=black,inner sep=1.5pt]{};\node at (-6,5) [circle,fill=black,inner sep=1.5pt]{};\node at (-5,6) [circle,fill=black,inner sep=1.5pt]{};\node at (-4,7) [circle,fill=black,inner sep=1.5pt]{};\node at (-3,8) [circle,fill=black,inner sep=1.5pt]{};\node at (-2,9) [circle,fill=black,inner sep=1.5pt]{};\node at (-1,10) [circle,fill=black,inner sep=1.5pt]{};\node at (0,10) [circle,fill=black,inner sep=1.5pt]{};
		\end{tikzpicture}
		\caption{Conjectured set that approximately realizes $\mathfs{M}_n(\mathbb{Z}^2)$\label{fig:spiral}}
	\end{figure}

	The proof of (\ref{1.5}) will be presented in an upcoming paper \cite{work-in-progress}. Note that $\mathbb{H}_{D_n}((0,0))=(2+\sqrt{3})^{-2n+o(n)}$ implies $\liminf\limits_{n\to \infty} -\frac{1}{n}\log(\mathfs{M}_n(\mathbb{Z}^2))\ge 2\log(2+\sqrt{3})$. Combined with (\ref{lowerbound_psi}), this yields
	\begin{equation*}
		2\log(2+\sqrt{3})	\le 	\liminf\limits_{n\to \infty} -\frac{1}{n}\log(\mathfs{M}_n(\mathbb{Z}^2))\le \limsup\limits_{n\to \infty} -\frac{1}{n}\log(\mathfs{M}_n(\mathbb{Z}^2)) \le \log(\psi(\mathbb{Z}^2)).
	\end{equation*}
	As a result, to establish (\ref{1.5}), it suffices to prove that $\psi(\mathbb{Z}^2)\le (2+\sqrt{3})^2$ (and therefore $\psi(\mathbb{Z}^2)= (2+\sqrt{3})^2$). However, to our knowledge, verifying the exact value of $\psi(\mathbb{Z}^2)$ is significantly more challenging than proving (\ref{1.5}). We hereby record the following conjecture regarding $\psi(\mathbb{Z}^2)$, which coincides with (\ref{1.5}). 
	\begin{conjecture}
		$\psi(\mathbb{Z}^2)=(2+\sqrt{3})^2$. 
	\end{conjecture}
	
	Though for $\mathbb{Z}^d$ with $d\ge 3$, $\psi(\mathbb{Z}^d)=\infty$, and thus the aformentioned argument for Theorem \ref{coro1} fails, we can demonstrate the exponential decay for $\mathfs{M}_n(\mathbb{Z}^d)$ by adapting the argument to the removal of an entire $*$-connected component.
	\begin{theorem}\label{thm2}
		In the case of $\mathbb{Z}^d$ for $d\ge 3$, there exist constants $C_2(d),c_2(d)>0$ such that for any integer $n\ge 2$, 
		\begin{equation}\label{1.6}
			e^{-C_2n}\le \mathfs{M}_n(\mathbb{Z}^d) \le e^{-c_2n}. 
		\end{equation}
	\end{theorem}

As for the set realizing $\mathfs{M}_n(\mathbb{Z}^d)$ ($d\ge 3$), it is worth noting that a very sparse set, such as the one presented in Figure \ref{fig:2^n}, will no longer yield an exponential decay for the harmonic measure. This is primarily due to the transience of $\mathbb{Z}^d$ for $d\ge 3$, which causes the vanishing of the hitting probability at a faraway vertex for the random walk. This observation also motivates us to present the following conjecture: (the definition of ``$*$-connected'' can be found in Section \ref{subsection_Zd})

\begin{conjecture}\label{conj_*connect}
	In the case of $\mathbb{Z}^d$ with $d\ge 3$, for any integer $n\ge 2$, the sets achieving $\mathfs{M}_n(\mathbb{Z}^d)$ are $*$-connected.
\end{conjecture}

Another interesting followup to this paper, is to study how the least positive value of harmonic measures changes when we impose a minimum separation distance among the vertices. Addressing this question will help us understand how the connectivity of the set influences the extremum of harmonic measures, thereby facilitating the resolution of Conjecture \ref{conj_*connect}. As shown in the following conjecture, we expect that as the required minimum distance increases, the least positive value of harmonic measures transitions from exponential to stretched exponential, and ultimately to polynomial.

\begin{conjecture}\label{conj_3}
For any $d\ge 3$, and integers $n\ge 2$ and $m\ge 3$, define
\begin{equation}
	\mathfs{M}_n^m(\mathbb{Z}^d):= \inf_{A\subset \mathbb{Z}^d,y\in A:\forall x_1,x_2\in A, \|x_1-x_2\|\ge m} \mathbb{H}_A(y),
\end{equation}
where $\|x_1-x_2\|$ is the graph distance between $x_1$ and $x_2$. Then we have  

	\begin{enumerate}
		\item There exist $C(d,m),c(d,m)>0$ such that for all $n\ge 2$,
		\begin{equation}
		 e^{-C n^{1/d}}	\le \mathfs{M}_n^m(\mathbb{Z}^d)\le e^{-c n^{1/d}}	. 
		\end{equation}

		\item When $m\ge f(n)$ (for some increasing function $f(n)$), there exist constants $C(d),c(d)>0$ such that for all $n\ge 2$, 
		\begin{equation}
	cn^{-1}	\le \mathfs{M}_n^m(\mathbb{Z}^d)\le Cn^{-1}.  
		\end{equation}

	\end{enumerate}
\end{conjecture}

Here are some explanations for Conjecture \ref{conj_3}. For Item (1), intuitively, for any $y\in \mathbb{Z}^d$, the vertices faraway from $y$ have less influence on the harmonic measure at $y$. Therefore, we expect that $\mathfs{M}_n^m(\mathbb{Z}^d)$ might be realized by arranging $A$ into a ball centering at some $y$ as densely as possible, for which the harmonic measure at $y$ is roughly $e^{-O (n^{1/d})}$. For Item (2), suppose that the $n$ vertices in $A$ are sufficiently spaced apart, then the harmonic measure of $A$ should approximate the uniform distribution on $A$, resulting in $\frac{1}{n}$ for the value at each vertex by symmetry. The interesting question is to identify an optimal estimate on the order of the function $f(n)$ in Item (2).

We conclude this section by providing an overview of the organization of this paper. In Section \ref{section_notation}, we present necessary notations, and review useful results on graph theory and random walks. In Section \ref{section_psi_finite}, we demonstrate that $\psi(\mathbb{Z}^2)<\infty$. Subsequently, Section \ref{section_psi_infinite} contains the proof of $\psi(\mathbb{Z}^d)=\infty$ for $d\ge 3$. Lastly, we establish Theorem \ref{thm2} in Section \ref{section_3d_decay}.


	\section{Preliminaries}\label{section_notation}

\subsection{Notations for $\mathbb{Z}^d$}\label{subsection_Zd}

We first present some notations for the lattice $\mathbb{Z}^d$.

\begin{itemize}
	\item  \textbf{Origin}: We denote the origin of $\mathbb{Z}^d$ by $\bm{0}$. Recall the notation $\rho_{A,y}(z)$ in (\ref{def_rho}). When $y=\bm{0}$, we may omit the subscript ``$y$'' in $\rho_{A,y}(z)$ and denote $\rho_{A}(z):=\rho_{A,\bm{0}}(z)$.

	\item  \textbf{Equivalent definitions of $\psi$ and $\mathcal{M}_n$:}  Since $\mathbb{Z}^d$ is a vertex-transitive graph, the harmonic measure on $\mathbb{Z}^d$ is translation-invariant. I.e.,  
	\begin{equation*}\label{tran_inv}
	\ \ \ \ \ \ \	\mathbb{H}_A(y) = \mathbb{H}_{A-y}(\bm{0}), \ \ \forall A\subset \mathbb{Z}^d\ \text{and}\ y\in A,
	\end{equation*}
	where $A-y:= \{z-y: z\in A\}$. Therefore, $\psi(\mathbb{Z}^d)$ in Definition \ref{def1} can be equivalently defined as
	\begin{equation}
			\psi(\mathbb{Z}^d)= \sup_{A\in \mathcal{A}(\mathbb{Z}^d)} \min_{z\in A\setminus \{\bm{0}\}}  \rho_{A}(z),
	\end{equation}
	where $\mathcal{A}(\mathbb{Z}^d):= \{A\subset \mathbb{Z}^d:\mathbb{H}_A(\bm{0})>0,|A|\ge 2\}$. For the same reason, $\mathcal{M}_n(\mathbb{Z}^d)$ (recall (\ref{def_Mn})) can be equivalently written as 
	\begin{equation}
		\mathcal{M}_n(\mathbb{Z}^d)=\inf_{A\in \mathcal{A}(\mathbb{Z}^d):|A|=n} \mathbb{H}_A(\bm{0}).
	\end{equation}

	
	
	\item  \textbf{Edge set:} We denote the edge set of $\mathbb{Z}^d$ by $$\mathbb{L}^d:=\{\{x,y\}:x\ \text{and}\ y\ \text{are adjacent vertices in}\ \mathbb{Z}^d\}.$$

	\item \textbf{Ball and box:} Recall that we use $\|\cdot \|$ to represent the graph distance, which is equivalent to the $\ell^1$ distance in $\mathbb{Z}^d$. For any $x\in \mathbb{Z}^d$ and $R\ge 0$, we denote the $\ell^1$ (resp.  $\ell^2$) ball with center $x$ and radius $R$ by $\mathbf{B}_x(R):=\{y\in\mathbb{Z}^d: \|x-y\|\le R\}$ (resp. $B_x(R):=\{y\in\mathbb{Z}^d: |x-y|\le R\}$). Let $\Lambda_{x}(R):= \{y\in\mathbb{Z}^d: |x-y|_{\infty}\le R \}$ (where $|\cdot|_{\infty}$ is the $\ell^\infty$ norm) be the box with center $x$ and side length $2\lfloor R \rfloor$. Especially, when $x=\bm{0}$, we may omit the subscript and denote $\mathbf{B}(R):=\mathbf{B}_{\bm{0}}(R)$, $B(R):=B_{\bm{0}}(R)$ and $\Lambda(R):=\Lambda_{\bm{0}}(R)$.


	\item  \textbf{Path:} We denote the set of non-negative (resp. positive) integers by $\mathbb{N}^0$ (resp. $\mathbb{N}^+$). A path $\eta$ on $\mathbb{Z}^d$ is a sequence of vertices $(x_0,...,x_n)$ (where $n\in \mathbb{N}^0$) such that $x_i\sim x_{i+1}$ for all $0\le i\le n-1$. We denote the length of $\eta$ by $\boldsymbol{\mathrm{L}}(\eta):=n$. The range of $\eta$ is $\boldsymbol{\mathrm{R}}(\eta):=\{x_i\}_{i=0}^{n}$. We write $\eta(i):=x_i$ for $0\le i\le n$. For convenience, we also write $\eta(-1):=x_n$ for the last vertex of $\eta$. Note that $(x_0)$ is a path with length $0$ and range $\{x_0\}$ and in addition, its first and last vertex are both $x_0$.

	\item  \textbf{Concatenation:} For any paths $\eta_1,\eta_2$ such that $\eta_1(-1)=\eta_2(0)$, we denote the concatenation of $\eta_1$ and $\eta_2$ by $\eta_1\circ \eta_2:= \big(\eta_1(0),...,\eta_1(-1),\eta_2(1),...,\eta_2(-1)\big)$.

	\item  \textbf{Self-avoiding path:} We say $\eta$ is self-avoiding if $\eta(i)\neq \eta(j)$ for all $i\neq j$.

	\item  \textbf{Circuit:} A circuit is a path $\eta$ such that for any $0\le i<j\le \boldsymbol{\mathrm{L}}(\eta)$, 
	$$
	\eta(i)=\eta(j) \iff i=0\ \text{and}\ j=\boldsymbol{\mathrm{L}}(\eta).
	$$

	\item  \textbf{Graph distance:} For any $A_1,A_2,F\subset \mathbb{Z}^d$, the graph distance between $A_1$ and $A_2$ in $F$ is defined as $\boldsymbol{\mathrm{D}}^F(A_1,A_2):=\inf\{n\in \mathbb{N}^0:\exists\ \text{path}\ \eta\ \text{with}\ \eta(0)\in A_1,\eta(-1)\in A_2,\boldsymbol{\mathrm{R}}(\eta)\subset F\ \text{and}\ \boldsymbol{\mathrm{L}}(\eta)=n\}$ (recalling that $\inf\emptyset=\infty$). Moreover, we say $A_1$ and $A_2$ are connected by $F$ if $\boldsymbol{\mathrm{D}}^F(A_1,A_2)<\infty$. When $F=\mathbb{Z}^d$, we may omit the superscript and denote $\boldsymbol{\mathrm{D}}(A_1,A_2):=\boldsymbol{\mathrm{D}}^{\mathbb{Z}^d}(A_1,A_2)$. When $A_i=\{x\}$ for some $i\in \{1,2\}$ and $x\in \mathbb{Z}^d$, we may omit the braces.

	\item  \textbf{Conectivity:} For any $A,F\subset \mathbb{Z}^d$, we say that $A$ is connected in $F$ if $\boldsymbol{\mathrm{D}}^{A\cap F}(x_1,x_2)<\infty$ for all $x_1,x_2\in A$. When $F=\mathbb{Z}^d$, we may omit the description ``in $F$''. For convenience, we define $\emptyset$ to be connected in any subset of $\mathbb{Z}^d$ (including $\emptyset$).

	\item  \textbf{Graph diameter:} For any connected $A$, the graph diameter of $A$ is defined as  $\mathrm{diam}(A):=\max_{x_1,x_2\in A}\boldsymbol{\mathrm{D}}^A(x_1,x_2)$.

	\item  \textbf{Connected component (cluster):} For any non-empty $A'\subset A\subset \mathbb{Z}^d$, we say $A'$ is a connected component of $A$ if $A'$ is connected and $\boldsymbol{\mathrm{D}}^A(x,A')=\infty$ for all $x\in A\setminus A'$. We also use ``cluster'' as a synonym of ``connected component''.

	\item  \textbf{Two specific clusters:} For any finite $A\subset \mathbb{Z}^d$, we denote by $A^c:=\mathbb{Z}^d \setminus A$ the complement of $A$. Note that $A^c$ contains exactly one infinite cluster, which we denote by $A^c_{\infty}$. Moreover, for any sufficiently large $R>0$ with $A\subset \Lambda(R)$, one has $[\Lambda(R)]^c\subset A^c_{\infty}$. We also denote by $A^c_{\bm{0}}$ the cluster of $A^c\cup \{\bm{0}\}$ that contains $\bm{0}$.

	\item  \textbf{Neighborhood:} For any $x\in \mathbb{Z}^d$, we denote the neighborhood of $x$ by $N(x):=\{y\in \mathbb{Z}^d:y\sim x\}$.  For any $y\in A$, let $N^A(y):=N(y)\cap  A^c$ be the neighborhood of $y$ outside $A$. Then we define the exterior (resp. $\bm{0}$-exposed) neighborhood of $y$ outside $A$ as $N^A_{\infty}(y):=N^A(y)\cap A^c_{\infty}$ (resp. $N^A_{\bm{0}}(y):=N^A(y)\cap A^c_{\bm{0}}$).

	\item  \textbf{Boundary:} For any finite $A\subset \mathbb{Z}^d$, the outer (resp. inner) vertex boundary of $A$ is defined as $\partial^{\mathrm{o}} A:= \cup_{y\in A} N^A(y)$ (resp. $\partial^{\mathrm{i}} A:=\{y\in A:N^A(y)\neq \emptyset\}$). We define the exterior (resp. $\bm{0}$-exposed) outer vertex boundary of $A$ by $\partial_{\infty}^{\mathrm{o}} A:= \cup_{y\in A} N^A_\infty(y)$ (resp. $\partial_{\bm{0}}^{\mathrm{o}} A:= \cup_{y\in A} N^A_{\bm{0}}(y)$). Moreover, the exterior (resp. $\bm{0}$-exposed) inner vertex boundary is defined as $\partial_\infty^{\mathrm{i}} A:=\{y\in A:N_\infty^A(y)\neq \emptyset\}$ (resp. $\partial_{\bm{0}}^{\mathrm{i}} A:=\{y\in A:N_{\bm{0}}^A(y)\neq \emptyset\}$).

	\item  \textbf{$*$-Adjacency:} For any $x,y\in \mathbb{Z}^d$, we say $x$ and $y$ are $*$-adjacent (denoted by $x\sim_* y$) if $|x-y|_{\infty}=1$. By replacing ``$\sim$'' with ``$\sim_*$'' in the concepts presented above, we obtain the definitions of $*$-path, $*$-circuit, $*$-graph distance ($\boldsymbol{\mathrm{D}}^F_{*}$, $\boldsymbol{\mathrm{D}}_{*}$), $*$-connected sets, $*$-graph diameter ($\mathrm{diam}_*$), $*$-cluster, $*$-neighborhoods ($N_*$, $N^A_*$, $N^A_{\infty,*}$ and $N^A_{\bm{0},*}$) and $*$-boundaries ($\partial^{\mathrm{o}}_*$, $\partial^{\mathrm{i}}_*$, $\partial^{\mathrm{o}}_{\infty,*}$, $\partial^{\mathrm{o}}_{\bm{0},*}$, $\partial^{\mathrm{i}}_{\infty,*}$ and $\partial^{\mathrm{i}}_{\bm{0},*}$).

	\item  \textbf{Embedding into $\mathbb{R}^d$:} We denote the canonical embedding from $\mathbb{Z}^d$ to $\mathbb{R}^d$ by $\boldsymbol{\mathrm{I}}(\cdot)$. For any edge $e=\{x,y\}\in \mathbb{L}^d$, we define the image of $e$ under $\boldsymbol{\mathrm{I}}$ (denoted by $\boldsymbol{\mathrm{I}}(e)$) as the line segment on $\mathbb{R}^d$ with endpoints $\boldsymbol{\mathrm{I}}(x)$ and $\boldsymbol{\mathrm{I}}(y)$. For any path $\eta=(x_0,...,x_n)$, the image of $\eta$ under $\boldsymbol{\mathrm{I}}$ (denoted by $\boldsymbol{\mathrm{I}}(\eta)$) is defined as follows: when $n=0$, $\boldsymbol{\mathrm{I}}(\eta):=\boldsymbol{\mathrm{I}}(x_0)$; when $n\ge 1$, $\boldsymbol{\mathrm{I}}(\eta):=\cup_{i=0}^{n-1}\boldsymbol{\mathrm{I}}(\{x_i,x_{i+1}\})$.

	\item  \textbf{Face:} In the case of $\mathbb{Z}^2$, $\mathbb{R}^2$ is divided by $\cup_{e\in \mathbb{L}^2}\boldsymbol{\mathrm{I}}(e)$ into connected regions, each of which is an open unit square and is called a face. For each face $\mathcal{S}$, we denote the collection of edges (resp. vertices) surrounding $\mathcal{S}$ by $\boldsymbol{\mathrm{e}}(\mathcal{S})$ (resp. $\boldsymbol{\mathrm{v}}(\mathcal{S})$). We say two faces $\mathcal{S}_1$ and $\mathcal{S}_2$ are adjacent (denoted by $\mathcal{S}_1\sim \mathcal{S}_2$) if $\boldsymbol{\mathrm{e}}(S_1)\cap \boldsymbol{\mathrm{e}}(S_2)\neq \emptyset$.

	\item  \textbf{Connectivity of faces:} We say a collection $\mathbb{S}$ of faces is connected if for any $\mathcal{S},\mathcal{S}'\in \mathbb{S}$, there exists a sequence of faces $\mathcal{S}_1,...,\mathcal{S}_k$ in $\mathbb{S}$ such that $\mathcal{S}_1=\mathcal{S}$, $\mathcal{S}_k=\mathcal{S}'$, and $\mathcal{S}_i\sim \mathcal{S}_{i+1}$ for all $1\le i \le k-1$.

\end{itemize}

\subsection{Connectivity of boundaries:} 
We cite the following lemma concerning the connectivity of different types of boundaries. 

\begin{lemma}\label{lemma_connectivity}
	For any finite, $*$-connected $D\subset \mathbb{Z}^d$, 
	\begin{enumerate}
		
		\item (\cite[Lemma 2.1]{deuschel1996surface}) $\partial_{\infty}^{\mathrm{o}} D$ is $*$-connected; 
		
		\item (\cite[Lemma 2.23]{kesten1984aspects}) $\partial_{\infty,*}^{\mathrm{o}} D $ is connected.  
	\end{enumerate}
\end{lemma}

In the case of $\mathbb{Z}^2$, the next lemma provides slightly stronger connectivity for $\partial_{\infty}^{\mathrm{o}} D$ (compared to Item (1) in Lemma \ref{lemma_connectivity}) under certain additional conditions.

\begin{lemma}\label{lemma_connectivity_new}
	For any finite, connected $D\subset \mathbb{Z}^2$, denote $\overline{D}:= D\cup \partial^{\mathrm{o}}_{\infty} D$. If $\partial_{\infty}^{\mathrm{i}} \overline{D} = \partial^{\mathrm{o}}_{\infty} D$, then $(\partial^{\mathrm{o}}_{\infty} D)\setminus \{z\}$ is $*$-connected for all $z \in \partial^{\mathrm{o}}_{\infty} D$.
\end{lemma}

\begin{proof}
By \cite[Corollary 2.2]{kesten1982percolation}, there exists a $*$-circuit $\eta_{*}$ with $\eta_{*}(0)=z$ and $\boldsymbol{\mathrm{R}}(\eta_{*})\subset \partial^{\mathrm{o}}_{\infty} D$ such that $D$ is contained in the interior of $\boldsymbol{\mathrm{R}}(\eta_{*})$ (i.e. the union of all finite clusters of $[\boldsymbol{\mathrm{R}}(\eta_{*})]^c$). In what follows, we prove $\boldsymbol{\mathrm{R}}(\eta_{*})=  \partial^{\mathrm{o}}_{\infty} D$ by contradiction. Assume that $w\in (\partial^{\mathrm{o}}_{\infty} D)\setminus \boldsymbol{\mathrm{R}}(\eta_{*})$. On the one hand, when $w$ is in the interior of $\boldsymbol{\mathrm{R}}(\eta_{*})$, any path from $w$ to infinity must intersect $\boldsymbol{\mathrm{R}}(\eta_{*})\subset (\partial^{\mathrm{o}}_{\infty} D)\setminus \{w\}$, which implies $w\notin \partial_{\infty}^{\mathrm{i}} \overline{D} $ and in turn causes contradiction with the assumption $\partial_{\infty}^{\mathrm{i}} \overline{D} = \partial^{\mathrm{o}}_{\infty} D$. On the other hand, when $w\in [\boldsymbol{\mathrm{R}}(\eta_{*})]^c_{\infty}$, since $D$ is contained in the interior of $\boldsymbol{\mathrm{R}}(\eta_{*})$, we know that any path from $w$ to $D$ must intersect $\boldsymbol{\mathrm{R}}(\eta_{*})$ and therefore has a length of at least $2$. However, this is incompatible with the fact that $\boldsymbol{\mathrm{D}}(w,D)=1$ (since $w\in \partial^{\mathrm{o}}_{\infty}D$). To sum up, we obtain $\boldsymbol{\mathrm{R}}(\eta_{*})=  \partial^{\mathrm{o}}_{\infty} D$. Consequently, each two vertices in $(\partial^{\mathrm{o}}_{\infty} D)\setminus \{z\}$ are connected by some sub-path of $\eta_{*}$ that does not intersect $z$ (recall that $\eta_*$ is a $*$-circuit). Now we conclude this lemma.
\end{proof}

\subsection{Statements about constants}

We use notations $C,C',c,c',...$ for the local constants whose values change according to the context. Additionally, we employ numbered notations $C_1, C_2, c_1, c_2,...$ to denote global constants. Unlike local constants, these global constants remain fixed throughout the paper. For clarity, we often assign the uppercase letter $C$ (possibly with subscripts or superscripts) to represent large constants, while the lowercase letter $c$ denotes small ones.

\subsection{Random walk}

	Recall that $\mathbb{P}_x$ for $x\in \mathbb{Z}^d$ is the law of the random walk $\{S_n\}_{n\ge 0}$ on $\mathbb{Z}^d$ starting from $x$. We denote the expectation under $\mathbb{P}_x$ by $\mathbb{E}_x$. Also recall that for any non-empty $A\subset \mathbb{Z}^d$, we denote $\tau_A:= \inf\{n\ge 0:S_n\in A\}$ and $\tau_A^+:=\inf\{n\ge 1:S_n\in A\}$. Especially, when $A=\{y\}$ for some $y\in \mathbb{Z}^d$, we may omit the braces.

	\textbf{Green's function:} For $d\ge 3$, the Green's function on $\mathbb{Z}^d$ is defined as 
	\begin{equation}\label{def_green}
		G(x,y):= \mathbb{E}_x\Big(\sum_{i=0}^{\infty} \mathbbm{1}_{S_i=y}\Big), \ \forall x,y\in \mathbb{Z}^d.   
	\end{equation}
	Note that for $\mathbb{Z}^2$, this definition is not valid since the expectation on the RHS of (\ref{def_green}) is infinity. However, for all $d\ge 2$, we may always define Green's function for a non-empty set $A\subset \mathbb{Z}^d$ by 
	\begin{equation}
		G_A(x,y):= \mathbb{E}_x\Big(\sum_{i=0}^{\tau_A} \mathbbm{1}_{S_i=y}\Big), \ \forall x,y\in A^c.   
	\end{equation}
	Next, we cite some basic formulas for $G_A(x,y)$.

\begin{lemma}[{\cite[Lemma 4.6.1 and Proposition 4.6.2]{lawler2010random}}]\label{lemma2.1}
	For any non-empty $A\subset \mathbb{Z}^d$ and $x,y\in A^c$, we have 
	\begin{equation}\label{2.2}
		G_A(x,y)= G_A(y,x), 
	\end{equation}
	\begin{equation}\label{2.3}
		G_A(x,x) = \left[ \mathbb{P}_x\left(\tau_A<\tau^+_x\right) \right]^{-1},  
	\end{equation}
	\begin{equation}\label{2.4}
		G_A(x,y)=  \mathbb{P}_x\left(\tau_y<\tau_A \right) G_A(y,y). 
	\end{equation}
In addition, for any non-empty $A'\subset A$, we have 
 \begin{equation}\label{2.5}
	G_{A'}(x,y)=G_{A}(x,y) + \sum_{w\in A\setminus A'} \mathbb{P}_x\left(\tau_{A}= \tau_w<\infty \right)G_{A'}(w,y). 
\end{equation}
\end{lemma}

We also need the following lemma, which is commonly referred to as ``the last-exit decomposition''. Its proof follows that of \cite[Proposition 4.6.4]{lawler2010random}. 
\begin{lemma}\label{lemma2.3}
	For any $A_1\subset A_2\subset \mathbb{Z}^d$, $y\in A_1$ and $z\in A_2\setminus A_1$, we have 
	\begin{equation}
		\mathbb{P}_{z}\left(\tau_{A_1}=\tau_y<\infty\right) = \sum_{v\in A_2\setminus A_1} G_{A_1}(z,v) \mathbb{P}_v\left( \tau_{A_2}^+= \tau_{y} <\infty\right).
	\end{equation}
\end{lemma}
\begin{proof}
	For the random walk starting from $z\in A_2\setminus A_1$, let $L$ the be largest integer in $\left[ 0,\tau_{A_1}\right)$ such that $S_L\in A_2$. Then we have that $\mathbb{P}_{z}\left(\tau_{A_1}=\tau_y<\infty\right)$ equals to 
	\begin{equation*}\label{lasttime_1}
		\begin{split}
	& \sum_{v\in A_2\setminus A_1}\sum_{k=0}^{\infty}  \mathbb{P}_{z}\left(\tau_{A_1}=\tau_y<\infty, L=k,S_L=v \right)\\
		&=\sum_{v\in A_2\setminus A_1}\sum_{k=0}^{\infty} \mathbb{P}_{z}  \left(k<\tau_{A_1}<\infty,S_k=v, S_{\tau_{A_1}}=y,\forall k+1\le j<\tau_{A_1},S_j\notin A_2  \right). 
		\end{split}
	\end{equation*}
	By Markov property, the probability on the RHS can be written as 
	\begin{equation*}
		\begin{split}
			&\mathbb{P}_{z} \left(k<\tau_{A_1},S_k=v\right)\cdot \mathbb{P}_{v}\left(\tau_{A_1}<\infty, S_{\tau_{A_1}}=y,\forall 1\le j<\tau_{A_1},S_j\notin A_2 \right)\\
			&= \mathbb{P}_{z} \left(k<\tau_{A_1},S_k=v\right)\cdot \mathbb{P}_{v}\left(\tau_{A_2}^+=\tau_y<\infty \right).
		\end{split}
	\end{equation*}
	Combining these two equations, we conclude this lemma.
\end{proof}

We then prove the subsequent lemma, which will facilitate the comparison between
Green’s functions for random walks with different starting points.
\begin{lemma}\label{lemma_compare_green}
	For any $d\ge 2$ and positive numbers $\lambda_1<\lambda_2<\lambda_3$, there exist constants $C_3(\lambda_1,\lambda_2,\lambda_3,d),c_3(\lambda_1,\lambda_2,\lambda_3,d)>0$ such that for any $N>C_3$, $A\subset \mathbf{B}(\lambda_1 N)\cup [\mathbf{B}(\lambda_3N)]^c$, $x_1,x_2\in \partial^{\mathrm{i}}\mathbf{B}(\lambda_2N)$ and $x_3\in  \partial^{\mathrm{i}}\mathbf{B}(\lambda_3N)$, 
	\begin{equation}\label{2.7}
		G_A(x_1,x_3)\ge c_3\cdot G_A(x_2,x_3). 
	\end{equation}
\end{lemma}
\begin{proof}
	Let $\epsilon=\frac{1}{4\sqrt{d}}\cdot \min\{\lambda_3-\lambda_2,\lambda_2-\lambda_1\}$. Note that for any $x_2\in \partial^{\mathrm{i}}\mathbf{B}(\lambda_2N)$, the $\ell^2$ ball $B_{x_2}(\epsilon N)$ is disjoint from $\mathbf{B}(\lambda_1 N)\cup [\mathbf{B}(\lambda_3N)]^c$. By strong Markov property, 
	\begin{equation}\label{2.8}
		\begin{split}
			G_{A}(x_1, x_3)\ge& \mathbb{P}_{x_1} \left(\tau_{\partial^{\mathrm{i}} B_{x_2}(\epsilon N)}<\tau_{\mathbf{B}(\lambda_1 N)\cup [\mathbf{B}(\lambda_3N)]^c} \right)  \cdot \min_{w\in \partial^{\mathrm{i}} B_{x_2}(\epsilon N)} G_{A}(w, x_3)\\
			\ge &c\cdot  \min_{w\in \partial^{\mathrm{i}} B_{x_2}(\epsilon N)} G_{A}(w, x_3), 
		\end{split}		 
	\end{equation}
	where we use the invariance principle in the second inequality. For each $w\in \partial^{\mathrm{i}} B_{x_2}(\epsilon N)$, since $A\cup \{x_3\}\subset [B_{x_2}(2\epsilon N)]^c$, one has: 
	\begin{equation}
		G_{A}(w, x_3)= \sum_{ w'\in \partial^{\mathrm{o}}B_{x_2}(2\epsilon N)} \mathbb{P}_{w}\left(\tau_{\partial^{\mathrm{o}}B_{x_2}(2\epsilon N)}=\tau_{w'} \right) G_{A}(w',x_3),  
	\end{equation}
	\begin{equation}
		G_{A}(x_2, x_3)= \sum_{ w'\in \partial^{\mathrm{o}}B_{x_2}(2\epsilon N)} \mathbb{P}_{x_2}\left(\tau_{\partial^{\mathrm{o}}B_{x_2}(2\epsilon N)}=\tau_{w'} \right) G_{A}(w',x_3).  
	\end{equation}
	Furthermore, it follows from \cite[Lemma 6.3.7]{lawler2010random} that 
	\begin{equation}\label{2.11}
		\mathbb{P}_{w}\left(\tau_{\partial^{\mathrm{o}}B_{x_2}(2\epsilon N)}=\tau_{w'} \right)\ge c\cdot \mathbb{P}_{x_2}\left(\tau_{\partial^{\mathrm{o}}B_{x_2}(2\epsilon N)}=\tau_{w'} \right),\ \forall  w'\in \partial^{\mathrm{o}}B_{x_2}(2\epsilon N).
	\end{equation}
	Thus, combining (\ref{2.8})-(\ref{2.11}), we obtain the desired bound (\ref{2.7}). 
\end{proof}

\textbf{Hitting two distant vertices on $\mathbb{Z}^2$:} The next lemma shows that for the random walk on $\mathbb{Z}^2$, the probability of hitting a vertex before another one from a comparable distance is uniformly bounded away from $0$.

\begin{lemma}\label{lemma_hit_distant}
	In the case of $\mathbb{Z}^2$, for any $\lambda>1.1$, there exist constants $C_4>0$ and $c_4\in (0,1)$ such that for any $N>C_4$, $x_1\in \partial^{\mathrm{i}}\mathbf{B}(N)$ and $x_2\in \partial^{\mathrm{i}}\mathbf{B}(\lambda N)$, 
	\begin{equation}
		\mathbb{P}_{x_1}\left(\tau_{x_2}<\tau_{\bm{0}} \right)\ge c_4. 
	\end{equation}
\end{lemma}
\begin{proof}
		By strong Markov property, we have 
	\begin{equation}\label{2.13}
		\mathbb{P}_{x_1} \left( \tau_{x_2}<\tau_{\bm{0}} \right) \ge \mathbb{P}_{x_1} \left( \tau_{\partial^{\mathrm{i}}\mathbf{B}(100\lambda N)}<\tau_{\{\bm{0},x_2\}} \right)\cdot \min_{w\in \partial^{\mathrm{i}}\mathbf{B}(100\lambda N)} \mathbb{P}_w\left(\tau_{x_2}<\tau_{\bm{0}} \right).   
	\end{equation}
	\cite[Proposition 1.6.7]{lawler2013intersections} implies that for a random walk starting from $\partial^{\mathrm{i}}\mathbf{B}(N)$, the probability that hits $\bm{0}$ (or $x_2$) before $\partial^{\mathrm{i}}\mathbf{B}(100\lambda N)$ is of the same order as $[\log(N)]^{-1}$. Therefore, the first term on the RHS of (\ref{2.13}) is at least $1-C[\log(N)]^{-1}$. Moreover, by \cite[Proposition 6.5.4]{lawler2010random}, for a random walk starting from $\partial^{\mathrm{i}}\mathbf{B}(100\lambda N)$, the probability that hits $x_2$ before $\bm{0}$ is of the same order as $\mathbb{H}_{\{x_2,\bm{0}\}}(x_2)$, which equals to $\frac{1}{2}$ by symmetry. Hence, the second term on the RHS of (\ref{2.13}) is bounded from below by some $c\in (0,1)$. Combining the aforementioned estimates, we conclude this lemma.
\end{proof}

\textbf{Escape probability and capacity on $\mathbb{Z}^d$ for $d\ge 3$:} In this part, we assume $d\ge 3$. For any finite $A\subset \mathbb{Z}^d$ and $x\in A$, the escape probability of $A$ at $x$ is $\mathrm{Es}_A(x) := \mathbb{P}_x\left( \tau_A^+=\infty \right)$. The capacity of $A$ is defined as $\mathrm{cap}(A):= \sum_{x\in A}\mathrm{Es}_A(x)$. Note that $\mathrm{cap}(\cdot)$ is invariant under translation. By \cite[Equation (2.13)]{lawler2013intersections}, the harmonic measure of $A$ can be equivalently written as
\begin{equation}\label{2.14}
	\mathbb{H}_A(x)= \frac{\mathrm{Es}_A(x)}{\mathrm{cap}(A)}, \ \forall x\in A. 
\end{equation}
We cite some basic properties of the capacity as follows. 
\begin{lemma}[{\cite[Proposition 2.2.1]{lawler2013intersections}}]\label{lemma_2.5}
	For any finite $A_1,A_2\subset \mathbb{Z}^d$ with $d\ge 3$, 
	\begin{equation}
		\mathrm{cap}(A_1)+ \mathrm{cap}(A_2) \ge  \mathrm{cap}(A_1\cup A_2)+ \mathrm{cap}(A_1\cap A_2). 
	\end{equation}
	In addition, if $A_1\subset A_2$, then one has 
	\begin{equation}\label{2.16}
		\mathrm{cap}(A_1)\le \mathrm{cap}(A_2). 
	\end{equation}
\end{lemma}

The subsequent lemma indicates that for the random walk on $\mathbb{Z}^d$ for $d\ge 3$, the probability of never hitting a ball from a comparable distance is bounded away from $0$.

\begin{lemma}[{\cite[Proposition 6.4.2]{lawler2010random}}]\label{lemma_escape}
	For any $d\ge 3$ and $\lambda>1.1$, there exist constants $C_5(d)>0$ and $c_5(d)\in (0,1)$ such that for all $N>C_5$ and $x\in \partial^{\mathrm{i}}\mathbf{B}(\lambda N)$, 
	\begin{equation}
		\mathbb{P}_x\left(\tau_{\mathbf{B}(N)}=\infty \right) \ge c_5.
	\end{equation}
\end{lemma}

	\section{Vertex-removal stability for $\mathbb{Z}^2$}\label{section_psi_finite}
	
	In this section, we focus on the case of $\mathbb{Z}^2$, and aim to demonstrate that $\psi(\mathbb{Z}^2)<\infty$. In essence, our proof mainly relies on a geometric argument that guides us in making the correct removal. Specifically, in Section \ref{subsection_marginal} we define the marginal vertices (which serve as candidates that can be removed with a bounded price), and then we establish its existence in every $*$-connected set (see Lemma \ref{lemma_marginal}). Subsequently, Section \ref{subsection_decomposition} presents a useful upper bound for the price of removing a subset (see Lemma \ref{lemma_upper_bound}). With these preparations, we conclude $\psi(\mathbb{Z}^2)<\infty$ in Section \ref{subsection_proof_psi_finite}.

	\subsection{Existence of a marginal vertex}\label{subsection_marginal}
	
	We first introduce the definition of marginal vertices and $*$-cut vertices:
		
	\begin{definition}\label{definition_marginal_cut}
		For any $*$-connected $A\subset \mathbb{Z}^2$ and $z\in A$, 
		\begin{enumerate}[(a)]
			\item we say $z$ is a marginal vertex (of $A$) if $N^A_{\infty}(z)$ is connected in $N^A_{\infty,*}(z)$;

			\item we say $z$ is a $*$-cut vertex (of $A$) if $A\setminus \{z\}$ is not $*$-connected.

		\end{enumerate}

	\end{definition}
	
	In the example shown by Figure \ref{fig:marginal}, the set $A$ consists of all dots except the small cyan ones (which are the vertices in $N_{\infty}^A(z_i)$ for $i\in \{1,2,3,4\}$). Note that $A$ is not $*$-connected, and includes two $*$-clusters, denoted by $A_{(\bm{0})}$ (the one containing $\bm{0}$) and $A'$. It follows from Definition \ref{definition_marginal_cut} that $z_1$ (resp. $z_3$ and $x_4$) is a marginal vertex of $A'$ (resp. $A_{(\bm{0})}$), while $z_2$ is not a marginal vertex of $A'$.

	The following lemma is crucial to the proof of $\psi(\mathbb{Z}^2)<\infty$. 
	
	\begin{lemma}\label{lemma_marginal}
		For any non-empty, $*$-connected $A\subset \mathbb{Z}^2$, the following holds:
		\begin{enumerate}
			\item There exists a marginal vertex in $A$. 
			
			\item For any $*$-cut vertex $z\in A$, in every $*$-cluster of $A\setminus \{z\}$ there exists a marginal vertex of $A$. 
			
		\end{enumerate}

	\end{lemma}

	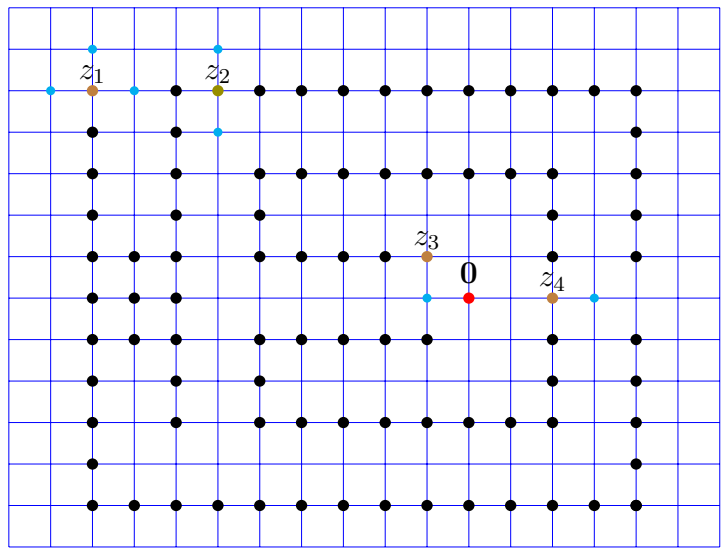
\begin{figure}[h!]
		\begin{tikzpicture}[scale=0.55]
			
			\draw[step=1cm,blue,ultra thin] (-10,-6) grid (7,7);
			\node at (1,0) [circle,fill=red,inner sep=1.5pt]{}; 
			\node at (1,0.6) []{$\bm{0}$};
			\node at (3,0) [circle,fill=brown,inner sep=1.5pt]{}; \node at (3,0.4) []{$z_4$};
			\node at (4,0) [circle,fill=cyan,inner sep=1.2pt]{}; 
			\node at (3,1) [circle,fill=black,inner sep=1.5pt]{}; 
			\node at (3,2) [circle,fill=black,inner sep=1.5pt]{};
			\node at (3,3) [circle,fill=black,inner sep=1.5pt]{};  
			\node at (2,3) [circle,fill=black,inner sep=1.5pt]{};  
			\node at (1,3) [circle,fill=black,inner sep=1.5pt]{};  
			\node at (0,3) [circle,fill=black,inner sep=1.5pt]{};  
			\node at (-1,3) [circle,fill=black,inner sep=1.5pt]{}; 
			\node at (-2,3) [circle,fill=black,inner sep=1.5pt]{}; 
			\node at (-3,3) [circle,fill=black,inner sep=1.5pt]{}; 
			\node at (-4,3) [circle,fill=black,inner sep=1.5pt]{}; 
			\node at (-4,2) [circle,fill=black,inner sep=1.5pt]{}; 
			\node at (-4,1) [circle,fill=black,inner sep=1.5pt]{}; 
			\node at (-3,1) [circle,fill=black,inner sep=1.5pt]{}; 
			\node at (-2,1) [circle,fill=black,inner sep=1.5pt]{}; 
			\node at (-1,1) [circle,fill=black,inner sep=1.5pt]{}; 
			\node at (0,1) [circle,fill=brown,inner sep=1.5pt]{}; \node at (0,1.4) []{$z_3$};
			\node at (0,0) [circle,fill=cyan,inner sep=1.2pt]{};

			\node at (-4,-1) [circle,fill=black,inner sep=1.5pt]{}; 
			\node at (-3,-1) [circle,fill=black,inner sep=1.5pt]{}; 
			\node at (-2,-1) [circle,fill=black,inner sep=1.5pt]{}; 
			\node at (-1,-1) [circle,fill=black,inner sep=1.5pt]{}; 
			\node at (0,-1) [circle,fill=black,inner sep=1.5pt]{}; 
			
			\node at (3,-2) [circle,fill=black,inner sep=1.5pt]{};
			\node at (3,-3) [circle,fill=black,inner sep=1.5pt]{};  
			\node at (2,-3) [circle,fill=black,inner sep=1.5pt]{};  
			\node at (1,-3) [circle,fill=black,inner sep=1.5pt]{};  
			\node at (0,-3) [circle,fill=black,inner sep=1.5pt]{};  
			\node at (-1,-3) [circle,fill=black,inner sep=1.5pt]{}; 
			\node at (-2,-3) [circle,fill=black,inner sep=1.5pt]{}; 
			\node at (-3,-3) [circle,fill=black,inner sep=1.5pt]{}; 
			\node at (-4,-3) [circle,fill=black,inner sep=1.5pt]{}; 
			\node at (-4,-2) [circle,fill=black,inner sep=1.5pt]{};
			\node at (3,-1) [circle,fill=black,inner sep=1.5pt]{};  
			
			\node at (5,1) [circle,fill=black,inner sep=1.5pt]{}; 
			\node at (5,2) [circle,fill=black,inner sep=1.5pt]{}; 
			\node at (5,3) [circle,fill=black,inner sep=1.5pt]{}; 
			\node at (5,4) [circle,fill=black,inner sep=1.5pt]{}; 
			\node at (5,5) [circle,fill=black,inner sep=1.5pt]{}; 
			\node at (5,-1) [circle,fill=black,inner sep=1.5pt]{}; 
			\node at (5,-2) [circle,fill=black,inner sep=1.5pt]{}; 
			\node at (5,-3) [circle,fill=black,inner sep=1.5pt]{}; 
			\node at (5,-4) [circle,fill=black,inner sep=1.5pt]{}; 
			\node at (5,-5) [circle,fill=black,inner sep=1.5pt]{}; 
			\foreach \x in {-8,...,5}
			{
				\node at (\x,-5) [circle,fill=black,inner sep=1.5pt]{}; 	
			}
			\foreach \y in {-4,...,4}
			{
				\node at (-8,\y) [circle,fill=black,inner sep=1.5pt]{}; 	
			}
			
			\node at (-7,0) [circle,fill=black,inner sep=1.5pt]{};
			\node at (-7,1) [circle,fill=black,inner sep=1.5pt]{};
			\node at (-7,-1) [circle,fill=black,inner sep=1.5pt]{};
			
			\node at (-8,5) [circle,fill=brown,inner sep=1.5pt]{}; \node at (-8,5.4) []{$z_1$};
			\node at (-8,6) [circle,fill=cyan,inner sep=1.2pt]{}; 
			\node at (-7,5) [circle,fill=cyan,inner sep=1.2pt]{}; 
			\node at (-9,5) [circle,fill=cyan,inner sep=1.2pt]{}; 
			\node at (-5,6) [circle,fill=cyan,inner sep=1.2pt]{}; \node at (-5,5.4) []{$z_2$};
			\node at (-5,4) [circle,fill=cyan,inner sep=1.2pt]{};

			
			\node at (4,5) [circle,fill=black,inner sep=1.5pt]{}; 
			\node at (3,5) [circle,fill=black,inner sep=1.5pt]{}; 
			\node at (2,5) [circle,fill=black,inner sep=1.5pt]{}; 
			\node at (1,5) [circle,fill=black,inner sep=1.5pt]{}; 
			\node at (0,5) [circle,fill=black,inner sep=1.5pt]{}; 
			\node at (-1,5) [circle,fill=black,inner sep=1.5pt]{}; 
			\node at (-2,5) [circle,fill=black,inner sep=1.5pt]{};
			\node at (-3,5) [circle,fill=black,inner sep=1.5pt]{};  
			\node at (-4,5) [circle,fill=black,inner sep=1.5pt]{}; 
			\node at (-5,5) [circle,fill=olive,inner sep=1.5pt]{}; 
			\node at (-6,5) [circle,fill=black,inner sep=1.5pt]{}; 
			
			\node at (-6,4) [circle,fill=black,inner sep=1.5pt]{}; 
			\node at (-6,3) [circle,fill=black,inner sep=1.5pt]{}; 
			\node at (-6,2) [circle,fill=black,inner sep=1.5pt]{}; 
			\node at (-6,1) [circle,fill=black,inner sep=1.5pt]{}; 
			\node at (-6,0) [circle,fill=black,inner sep=1.5pt]{}; 
			\node at (-6,-1) [circle,fill=black,inner sep=1.5pt]{}; 
			\node at (-6,-2) [circle,fill=black,inner sep=1.5pt]{}; 
			\node at (-6,-3) [circle,fill=black,inner sep=1.5pt]{};

		\end{tikzpicture}
			\caption{An illustration for marginal vertex \label{fig:marginal}}
	\end{figure}

	The proof of Lemma \ref{lemma_marginal} can be divided into the following two steps.
	
	\begin{lemma}\label{lemma_cut}
		For any non-empty, $*$-connected $A\subset \mathbb{Z}^2$, the following holds: 
		\begin{enumerate}
			\item There exists a vertex in $A$ that is not a $*$-cut vertex. 
			
			\item For any $*$-cut vertex $z\in A$, in every $*$-cluster of $A\setminus \{z\}$ there exists a vertex that is not a $*$-cut vertex of $A$.
		
		\end{enumerate}
	\end{lemma}
	\begin{proof}
		We prove these two items separately as follows. 
		
		(1) When $|A|\le 2$, it is obvious that every vertex in $A$ is not a $*$-cut vertex.

		When $|A|\ge 3$, there exist $z_1,z_2\in A$ such that the $*$-graph distance between $z_1$ and $z_2$ in $A$ realizes the $*$-graph diameter of $A$. I.e., $\boldsymbol{\mathrm{D}}^A_{*}(z_1,z_2)=\mathrm{diam}_{*}(A)$. In what follows, we prove that $z_2$ is not a $*$-cut vertex. To get this, it suffices to show that for any $y\in A\setminus \{z_1,z_2\}$, there is a $*$-path from $z_1$ to $y$ without intersecting $z_2$. If this is not true, then there exists $y_{\dagger}\in A\setminus \{z_1,z_2\}$ such that every $*$-path from $z_1$ must intersect $z_2$ before $y_{\dagger}$. This implies that $\boldsymbol{\mathrm{D}}^A_{*}(z_1,y_{\dagger})> \boldsymbol{\mathrm{D}}^A_{*}(z_1,z_2)=\mathrm{diam}_{*}(A)$, which is contradictory with the definition of $\mathrm{diam}_{*}(A)$. Now we conclude that $z_2$ is not a $*$-cut vertex and thus confirm Item (1).

		(2) Let $z$ be a $*$-cut vertex of $A$, and let $A'$ be an arbitrary $*$-cluster of $A\setminus \{z\}$. Assume that $\bar{w}$ is a vertex in $A'$ that maximizes $\boldsymbol{\mathrm{D}}^{A'\cup \{z\}}(\cdot, z)$. Subsequently, we prove that $\bar{w}$ is not a $*$-cut vertex of $A$ by using a similar argument as in Item (1). For this, it suffices to show that for any $w\in A\setminus \{z,\bar{w}\}$, there is a $*$-path $\eta_*$ from $z$ to $w$ without intersecting $\bar{w}$. In fact, for any $w$ in some $*$-cluster $A''$ of $A\setminus \{z\}$ other than $A'$, such a $*$-path $\eta_*$ exists since $A''\cup \{z\}$ is $*$-connected. Furthermore, assume that there is a counter-example $w_{\dagger}\in A'\setminus \{\bar{w}\}$ (i.e. every $*$-path from $z$ to $w_{\dagger}$ must intersect $\bar{w}$), then one has $\boldsymbol{\mathrm{D}}^{A'\cup \{z\}}(w_{\dagger}, z)>\boldsymbol{\mathrm{D}}^{A'\cup \{z\}}(\bar{w}, z)$, which is incompatible with the maximality of $\boldsymbol{\mathrm{D}}^{A'\cup \{z\}}(\bar{w}, z)$. By contradiction, such a counter-example does not exist and consequently, the proof is now complete.
	\end{proof}

	\begin{lemma}\label{lemma_cut_marginal}
		For any $*$-connected $A\subset \mathbb{Z}^2$ and $z\in A$, if $z$ is not a $*$-cut vertex of $A$, then $z$ is a marginal vertex of $A$.
	\end{lemma}
\begin{proof}
	
	When $|N_\infty^A(z)|\le 1$, it is obvious that $N_\infty^A(z)$ is connected in $N_{\infty,*}^A(z)$.

	When $|N_\infty^A(z)|\ge 2$, we arbitrarily take two vertices $v_1,v_2\in N_\infty^A(z)$. We also choose a sufficiently large integer $R$ such that $A\subset \Lambda(R)$, and then we arbitrarily take $x\in [\Lambda(R)]^c$. It follows from the definition of $N_\infty^A(z)$ that both $v_1$ and $v_2$ are connected to $x$ by $A^c$. I.e., for $i\in \{1,2\}$, there exists a path $\eta_i$ from $v_i$ to $x$ with $\boldsymbol{\mathrm{R}}(\eta_i)\subset A^c$. Let $\widetilde{\eta}:=\cup_{i=1}^2\boldsymbol{\mathrm{I}}(\{z,v_i\})\cup \boldsymbol{\mathrm{I}}(\eta_i)$. By planarity, $\mathbb{R}^2\setminus \widetilde{\eta}$ is composed of more than one connected component. In addition, among them there exist exactly two connected components (denoted by $\mathcal{C}_1$ and $\mathcal{C}_2$) that contains at least one face $\mathcal{S}$ with $z\in \boldsymbol{\mathrm{v}}(\mathcal{S})$ (since $\{z,v_1\}$ and $\{z,v_2\}$ are the only edges incident to $z$ that are traversed by $\widetilde{\eta}$). We claim that either $\mathcal{C}_1$ or $\mathcal{C}_2$ does not include any vertex in $A$. Otherwise, there exist $y_{1},y_{2}\in A\setminus \{z\}$ such that $y_i$ is contained in $\mathcal{C}_{i}$ for $i\in \{1,2\}$. Since $z$ is not a $*$-cut set, there exists a self-avoiding $*$-path $\eta_{*}$ such that $\eta_{*}(0)=y_1$, $\eta_{*}(-1)=y_2$ and $\boldsymbol{\mathrm{R}}(\eta_{*})\subset A\setminus \{z\}$. Note that for every $y\in \boldsymbol{\mathrm{R}}(\eta_{*})$, $\{\mathcal{S}:y\in \boldsymbol{\mathrm{v}}(\mathcal{S})\}$ is connected, and that for each $1\le i\le \boldsymbol{\mathrm{L}}(\eta_{*})-1$, $\{\mathcal{S}:\eta_{*}(i)\in \boldsymbol{\mathrm{v}}(\mathcal{S})\}\cap \{\mathcal{S}:\eta_{*}(i+1)\in \boldsymbol{\mathrm{v}}(\mathcal{S})\}\neq \emptyset$. Thus, the collection of faces $\{\mathcal{S}:\boldsymbol{\mathrm{v}}(\mathcal{S}) \cap \boldsymbol{\mathrm{R}}(\eta_{*})\neq \emptyset\}$ is connected, which indicates that  there exists a sequence of faces $\mathcal{S}_1,...,\mathcal{S}_k$ ($k\ge 2$) satisfying the following conditions:


	\begin{enumerate}
		\item $\mathcal{S}_1\subset \mathcal{C}_{1}$ and $\mathcal{S}_k\subset \mathcal{C}_{2}$;

		\item For any $0\le j\le k-1$, $\boldsymbol{\mathrm{e}}(\mathcal{S}_j)\cap \boldsymbol{\mathrm{e}}(\mathcal{S}_{j+1})$ contains an edge intersecting $A\setminus \{z\}$.

	\end{enumerate}
	By continuity, there exists $j_{\dagger}\in [1,k-1]$ such that $\mathcal{S}_{j_\dagger}\subset  \mathcal{C}_{1}$ and $\mathcal{S}_{j_\dagger+1}\not\subset \mathcal{C}_{1}$. It follows that the edge in   $\boldsymbol{\mathrm{e}}(\mathcal{S}_{j_\dagger})\cap \boldsymbol{\mathrm{e}}(\mathcal{S}_{j_\dagger+1})$ either is in $\big\{\{z,v_1\},\{z,v_2\} \big\}$, or is traversed by $\eta_1$ or $\eta_2$. Combined with the facts that $v_1,v_2\in N_{\infty}^A(z)\subset A^c $ and $\cup_{i=1}^2\boldsymbol{\mathrm{R}}(\eta_i)\subset A^c$, this implies that the edge in $\boldsymbol{\mathrm{e}}(\mathcal{S}_{j_\dagger})\cap \boldsymbol{\mathrm{e}}(\mathcal{S}_{j_\dagger+1})$ does not intersect $A\setminus \{z\}$, which is incompatible with Condition (2).  Now we conclude this claim by contradiction. Without loss of generality,
	assume that $\mathcal{C}_{1}$ does not contains any vertex in $A$. Let $\mathbb{S}_1^z$ be the collection of all faces $\mathcal{S}$ with $z\in \boldsymbol{\mathrm{v}}(\mathcal{S})$ and $\mathcal{S}\subset \mathcal{C}_{1}$. As shown in Figure \ref{fig:scenario_1} (where every face in $\mathbb{S}_1^z$ is colored in red), it follows from the construction that all edges $e\in \cup_{\mathcal{S}\in \mathbb{S}_1^z}\boldsymbol{\mathrm{e}}(\mathcal{S})$ with $z\notin e$ (i.e. red solid line segments) form a path $\eta_{\ddagger}$ from $v_1$ to $v_2$.

		\begin{figure}[h!]
		
		\begin{tikzpicture}[scale=0.55]
			

			
			\fill[fill=red!5] (-7,0) rectangle (-4,3);

			\node at (-7.4,3.2) [] {\color{red}$\boldsymbol{\eta_{\ddagger}}$};

			\node at (-5.5,1.5) [] {\color{red}$\boldsymbol{\mathbb{S}_1^z}$};

			\draw[blue,thick, dashed] (-4,0)--(-4,3);       
			
			\draw[blue,thick, dashed] (-4,0)--(-4,-3);

			\draw[blue,thick, dashed] (-4,0)--(-1,0);     
			
			\draw[blue,thick, dashed] (-4,0)--(-7,0);

			\draw[blue,thick, dashed] (-1,0)--(-1,3);       
			
			\draw[blue,thick, dashed] (-1,3)--(-4,3);

			\draw[blue,thick, dashed] (-1,0)--(-1,-3);      
			
			\draw[blue,thick, dashed] (-1,-3)--(-4,-3);     
			
			\draw[blue,thick, dashed] (-4,-3)--(-7,-3); 
			
			\draw[blue,thick, dashed] (-7,-3)--(-7,0);   
			
			

			\draw[->,red,thick] (-7,0)--(-7,1.5); 
			
			\draw[red,thick] (-7,1.5)--(-7,3); 
			
			\draw[->,red,thick] (-7,3)--(-5.5,3); 
			
			\draw[->,red,thick] (-5.5,3)--(-4,3); 
			
			\node at (-4,0) [circle,fill=black,inner sep=3pt]{};
			
			\node at (-3.5,-0.5) [] {$\bm{z}$};

			\node at (-4,3) [circle,fill=red,inner sep=3pt]{};
			
			\node at (-4,3.5) [] {$\bm{v_2}$};


			

			\node at (-7,0) [circle,fill=red,inner sep=3pt]{};
			
			\node at (-7.7,0) [] {$\bm{v_1}$};

%
%
%

%
%
%
%
%
%
%

			\begin{scope}[shift={(18,0)}]
				
				\fill[fill=red!5] (-7,-3) rectangle (-4,3);
				
				\fill[fill=red!5] (-4,-3) rectangle (-1,0);
				
				\node at (-7.4,3.2) [] {\color{red}$\boldsymbol{\eta_{\ddagger}}$};

				\node at (-5.5,-1.5) [] {\color{red}$\boldsymbol{\mathbb{S}_1^z}$};

				\draw[blue,thick, dashed] (-4,0)--(-4,3);       
				
				\draw[blue,thick, dashed] (-4,0)--(-1,0);

				\draw[blue,thick, dashed] (-4,0)--(-4,-3);    
				
				\draw[blue, thick,dashed] (-4,0)--(-7,0);

				\draw[blue,thick, dashed] (-1,0)--(-1,3);       
				
				\draw[blue,thick, dashed] (-1,3)--(-4,3);

				\draw[->,red,thick] (-1,0)--(-1,-1.5); 
				
				\draw[red,thick] (-1,-1.5)--(-1,-3);   
				
				\draw[red,thick, dashed] (-1,0)--(-1,-3);

				\draw[->,red,thick] (-1,-3)--(-2.5,-3); 
				
				\draw[red,thick] (-2.5,-3)--(-4,-3);

				\draw[->,red,thick] (-4,-3)--(-5.5,-3); 
				
				\draw[red,thick] (-5.5,-3)--(-7,-3); 
				
				\draw[->,red,thick] (-7,-3)--(-7,-1.5);  
				
				\draw[red,thick] (-7,-1.5)--(-7,0); 
				
				\draw[->,red,thick] (-7,0)--(-7,1.5);    
				
				\draw[red,thick] (-7,1.5)--(-7,3); 
				
				\draw[->,red,thick] (-7,3)--(-5.5,3);    
				
				\draw[red,thick] (-5.5,3)--(-4,3);

				\node at (-4,0) [circle,fill=black,inner sep=3pt]{};
				
				\node at (-3.5,-0.5) [] {$\bm{z}$};

				\node at (-4,3) [circle,fill=red,inner sep=3pt]{};
				
				\node at (-4,3.5) [] {$\bm{v_2}$};

				\node at (-1,0) [circle,fill=red,inner sep=3pt]{};
				
				\node at (-0.3,0) [] {$\bm{v_1}$};

%
%
%
%
%

%
%
%

%

%
%
%
%
%
%
%
%
%


%
%
%
%
%
%
%

			\end{scope}
			
			\begin{scope}[shift={(9,0)}]
				
				\fill[fill=red!5] (-7,-3) rectangle (-4,3);

				\node at (-7.4,3.2) [] {\color{red}$\boldsymbol{\eta_{\ddagger}}$};

				\node at (-5.5,0) [] {\color{red}$\boldsymbol{\mathbb{S}_1^z}$};

				\draw[blue,thick, dashed] (-4,0)--(-4,3);       
				
				\draw[blue,thick, dashed] (-4,0)--(-4,-3);

				\draw[blue,thick, dashed] (-4,0)--(-1,0);     
				
				\draw[blue, thick,dashed] (-4,0)--(-7,0);

				\draw[blue,thick, dashed] (-1,0)--(-1,3);       
				
				\draw[blue,thick, dashed] (-1,3)--(-4,3);

				\draw[blue,thick, dashed] (-1,0)--(-1,-3);      
				
				\draw[blue,thick, dashed] (-1,-3)--(-4,-3);

				\draw[->,red,thick] (-4,-3)--(-5.5,-3); 
				
				\draw[red,thick] (-5.5,-3)--(-7,-3);

				\draw[->,red,thick] (-7,-3)--(-7,-1.5); 
				
				\draw[red,thick] (-7,-1.5)--(-7,0);   
				
				\draw[->,red,thick] (-7,0)--(-7,1.5);   
				
				\draw[red,thick] (-7,1.5)--(-7,3);  
				
				\draw[->,red,thick] (-7,3)--(-5.5,3);    
				
				\draw[red,thick] (-5.5,3)--(-4,3);

				\node at (-4,0) [circle,fill=black,inner sep=3pt]{};
				
				\node at (-3.5,-0.5) [] {$\bm{z}$};

				\node at (-4,3) [circle,fill=red,inner sep=3pt]{};
				
				\node at (-4,3.5) [] {$\bm{v_2}$};

				\node at (-4,-3) [circle,fill=red,inner sep=3pt]{};
				
				\node at (-4,-3.5) [] {$\bm{v_1}$};

%

%

%
%
%

%
%
%
%
%
%
%
				
			\end{scope}

		\end{tikzpicture}
		
		\caption{Illustation for the path $\eta_{\ddagger}$ in different scenarios}
		
		\label{fig:scenario_1}
		
	\end{figure}
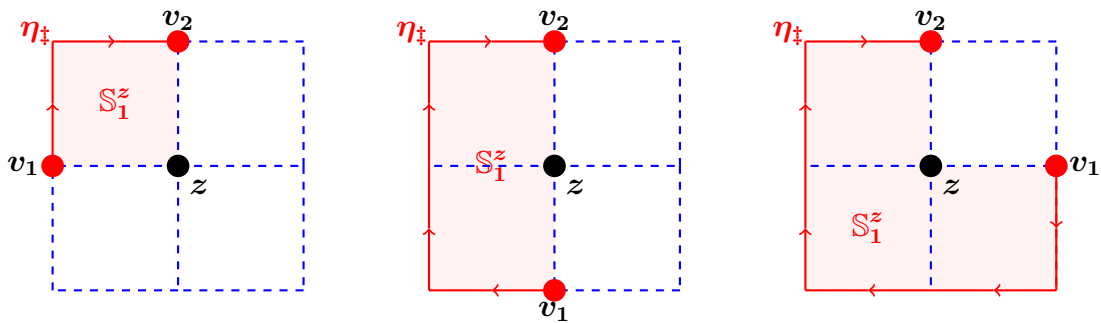

	Since $v_1,v_2\in N_\infty^A(z)$ and $\boldsymbol{\mathrm{R}}(\eta_{\ddagger})\subset N_{*}^A(z)$, one has $\boldsymbol{\mathrm{R}}(\eta_{\ddagger})\subset N_{\infty,*}^A(z)$. Therefore, $v_1$ and $v_2$ are connected by $N_{\infty,*}^A(z)$, thus concluding this lemma. \end{proof}

	Combining Lemmas \ref{lemma_cut} and \ref{lemma_cut_marginal}, we immediately obtain Lemma \ref{lemma_marginal}.

\begin{remark}\label{remark_psi_geometry}
	Lemmas \ref{lemma_marginal} and \ref{lemma_cut_marginal} fail for $\mathbb{Z}^d$ with $d\ge 3$. To see this, one may refer to the example named ``discrete Klein bottle'' in Figure \ref{fig:Klein}, where every vertex is neither a $*$-cut vertex nor a marginal vertex. This example suggests that the planarity of the graph is a key factor for the finiteness of $\psi$. Further discussions can be found in Remarks \ref{generalize_2d} and \ref{generalize_3d}.
\end{remark}

	\subsection{An upper bound for the price of removing a subset}\label{subsection_decomposition}

	To prove the finiteness of $\psi(\mathbb{Z}^2)$, we need to spatially decompose random walk trajectories. It is worth mentioning that this decomposition is valid for all $d\ge 2$.

	\begin{figure}[h!]
	\begin{tikzpicture}[scale=0.5]
		
		\draw[step=1cm,blue,thin, dotted] (-10,-6) grid (7,10);
		\node at (3,-3) [circle,fill=red,inner sep=1.5pt]{}; \node at (3.5,-3) []{$\bm{0}$};
		\node at (3,-2) [circle,fill=black,inner sep=1.5pt]{};       \node at (3, -4)[circle,fill=black,inner sep=1.5pt]{} ;\node at (2,-2) [circle,fill=black,inner sep=1.5pt]{}; \node at (4,-2) [circle,fill=black,inner sep=1.5pt]{};\node at (5,-3) [circle,fill=black,inner sep=1.5pt]{}; \node at (5,-4) [circle,fill=black,inner sep=1.5pt]{};\node at (4,-5) [circle,fill=black,inner sep=1.5pt]{};\node at (5,-5) [circle,fill=black,inner sep=1.5pt]{};\node at (6,-2) [circle,fill=black,inner sep=1.5pt]{};\node at (6,-1) [circle,fill=black,inner sep=1.5pt]{};\node at (6,0) [circle,fill=black,inner sep=1.5pt]{};\node at (6,-3) [circle,fill=black,inner sep=1.5pt]{};\node at (6,-4) [circle,fill=black,inner sep=1.5pt]{};
		\node at (6,-5) [circle,fill=black,inner sep=1.5pt]{};
		
		\node at (-1,4) [circle,fill=brown,inner sep=1.5pt]{};\node at (-1,5) [circle,fill=brown,inner sep=1.5pt]{};\node at (-2,4) [circle,fill=brown,inner sep=1.5pt]{};\node at (-2,5) [circle,fill=brown,inner sep=1.5pt]{};\node at (0,4) [circle,fill=brown,inner sep=1.5pt]{};\node at (0,5) [circle,fill=brown,inner sep=1.5pt]{};\node at (1,6) [circle,fill=black,inner sep=1.5pt]{};\node at (2,6) [circle,fill=black,inner sep=1.5pt]{};\node at (2,7) [circle,fill=black,inner sep=1.5pt]{};\node at (1,7) [circle,fill=black,inner sep=1.5pt]{};\node at (3,5) [circle,fill=black,inner sep=1.5pt]{};\node at (-3,3) [circle,fill=brown,inner sep=1.5pt]{};
		
		
		\draw[blue,thick, dashed] (0.5,4.5) arc (360:700:2.5);       
		\draw[red,thick, dashed] (0.5,4.5) arc (360:700:3.44);          
		
		\draw [magenta,thick] plot [smooth, tension=1.5] coordinates { (-10,7) (-9,6)(-9,8)(-6,7)(-7,9) (-4,6)};       
		
		\draw [magenta,thick] plot [smooth, tension=1] coordinates {(-6,3) (-7,2)(-3,-1)(-4,-2)(-6,-1)(-3,0)(-1,-2)(-1,-3.5)(0,-4)(2,-3)};
		\draw [magenta,thick] plot [smooth, tension=1] coordinates {(2,-3)(3,-3)};
		
		\draw [orange,thick] plot [smooth, tension=1] coordinates {(-4,6)(-3,5)(-3,7)(-2,6)};
	\draw [orange,thick] plot [smooth, tension=1] coordinates {(-2,6) (-2,4)};
	\draw [orange,thick] plot [smooth, tension=1] coordinates {(-2,4)(-3,4)};
	\draw [orange,thick] plot [smooth, tension=1] coordinates {(-3,4)(-4,3)(-6,3)};
	\draw [cyan,thick] plot [smooth, tension=1] coordinates {(-4,6)(-6,5)(-5,6)(-5,4)(-3,4)(-6,3)};
	
	\node at (-1,3) []{$\textcolor{brown}{D}$};
	\node at (3.5,2.5) []{$\widetilde{A}=A\setminus \textcolor{brown}{D}$};
	\node at (-4,6) [circle,fill=blue,inner sep=1pt]{};
	\node at (-4,6.5) []{$v_1$};
	\node at (-6,3) [circle,fill=blue,inner sep=1pt]{};
	\node at (-6,2.5) []{$v_2$};
	\node at (-1.5,6) []{\textcolor{blue}{$F_1$}};
	\node at (-3,7.4) []{\textcolor{red}{$F_2$}};
	\node at (-3.5,4.6) []{\textcolor{black}{$\eta_1$}};
	\node at (-5.6,4) []{\textcolor{black}{$\eta_2$}};

\end{tikzpicture}
\caption{An illustration for the decomposition. For the trajectories $\eta_1$, $\eta_2$, which are involved in $\mathbb{P}_x\left( \tau_D<\tau_{\widetilde{A}}=\tau_{\bm{0}}<\infty\right)$ and $\mathbb{P}_{v_1} \left(\tau_D<\tau_{\widetilde{A}}=\tau_{\bm{0}}<\infty \right)$ respectively, the pink paths are the common part, and the orange (resp. cyan) path belongs to $\eta_1$ (resp. $\eta_2$).}
\end{figure}

	Let $A\in \mathcal{A}(\mathbb{Z}^d)$ and $D\subset A\setminus \{\bm{0}\}$. Denote $\widetilde{A}:=A\setminus D$. We arbitrarily take $F_1\subset F_2\subset \mathbb{Z}^d$ such that $F_1\cap A=F_2\cap A=D$. For any $j\in \{1,2\}$, let $\widehat{F}_j:=[\partial^{\mathrm{i}}_{\infty} (A\cup F_j)]\setminus \widetilde{A}$ and $\widecheck{F}_j:=[\partial^{\mathrm{i}}_{\bm{0}} (A\cup F_j)]\setminus \widetilde{A}$. We choose a sufficiently large integer $R$ such that $A\cup F_2\subset \Lambda(R)$, and then we arbitrarily take $x\in [\Lambda(R)]^c$. In what follows, we present decompositions for $\mathbb{P}_x\left( \tau_D<\tau_{\widetilde{A}}=\tau_{\bm{0}}<\infty\right)$ and $\mathbb{P}_x\left( \tau_{F_1}<\tau_{A}=\tau_{\bm{0}}<\infty\right)$. By strong Markov property, one has 
	\begin{equation*}
		\begin{split}
			\mathbb{P}_x\left( \tau_D<\tau_{\widetilde{A}}=\tau_{\bm{0}}<\infty\right)=&\sum_{v_1\in \widehat{F}_1} \mathbb{P}_x\left( \tau_{A\cup F_1}=\tau_{v_1} <\infty\right) \mathbb{P}_{v_1} \left(\tau_D<\tau_{\widetilde{A}}=\tau_{\bm{0}}<\infty \right),   
		\end{split}
	\end{equation*}
	where we used the fact that for any $v_1\in A\cup F_1$, $$\mathbb{P}_x\left( \tau_{A\cup F_1}=\tau_{v_1} <\infty\right) \mathbb{P}_{v_1} \left(\tau_D<\tau_{\widetilde{A}}=\tau_{\bm{0}} <\infty\right)>0$$ if and only if $v_1\in \widehat{F}_1$. In addition, for any $v_1\in \widehat{F}_1$, we have 
	\begin{equation*}
		\mathbb{P}_{v_1} \left(\tau_D<\tau_{\widetilde{A}}=\tau_{\bm{0}} <\infty\right) = \sum_{w\in D} \mathbb{P}_{v_1} \left(\tau_A=\tau_w <\infty\right) \mathbb{P}_{w} \left(\tau_{\widetilde{A}}=\tau_{\bm{0}} <\infty\right). 
	\end{equation*}
	For any $w\in D$, by Lemma \ref{lemma2.3} (with $z=w$, $A_1=\widetilde{A}$, $A_2=A\cup F_2$ and $y=\bm{0}$),
	\begin{equation*}
		\mathbb{P}_{w} \left(\tau_{\widetilde{A}}=\tau_{\bm{0}} <\infty\right) = \sum_{v_2\in \widecheck{F}_2} G_{\widetilde{A}}(w,v_2) \mathbb{P}_{v_2}\left( \tau_{A\cup F_2}^+= \tau_{\bm{0}}<\infty \right), 
	\end{equation*}
	where we used the fact that for any $v_2\in F_2$, $G_{\widetilde{A}}(w,v_2) \mathbb{P}_{v_2}( \tau_{A\cup F_2}^+= \tau_{\bm{0}}<\infty )>0$ if and only if $v_2\in \widecheck{F}_2$. Combining these three equalities, we obtain 
	\begin{equation}\label{3.1}
		\begin{split}
			&\mathbb{P}_x\left( \tau_D<\tau_{\widetilde{A}}=\tau_{\bm{0}}<\infty\right)=\sum_{v_1\in \widehat{F}_1,v_2\in \widecheck{F}_2} \mathbb{P}_x\left(  \tau_{A\cup F_1}=\tau_{v_1} <\infty\right)\\
			&\hskip 2cm \cdot\mathbb{P}_{v_2}\left(  \tau_{A\cup F_2}^+= \tau_{\bm{0}} <\infty\right)
			   \Big[ \sum_{w\in D} \mathbb{P}_{v_1} \left( \tau_A=\tau_w <\infty\right)   G_{\widetilde{A}}(w,v_2)\Big].
		\end{split}	
	\end{equation} 
	With similar arguments, we also have 
	\begin{equation}\label{3.2}
		\begin{split}
			&\mathbb{P}_x\left( \tau_{F_1}<\tau_{A}=\tau_{\bm{0}}<\infty\right) \\
			&=\sum_{v_1\in \widehat{F}_1,v_2\in \widecheck{F}_2}  \mathbb{P}_x\left( \tau_{A\cup F_1}=\tau_{v_1} <\infty\right) \mathbb{P}_{v_2}\left(  \tau_{A\cup F_2}^+= \tau_{\bm{0}}<\infty \right)
			 G_{A}(v_1,v_2). 
		\end{split}
	\end{equation}
	By employing (\ref{3.1}) and (\ref{3.2}), we derive the following lemma, which will be used multiple times.
	\begin{lemma}\label{lemma_upper_bound}
		Keep the notations in the last paragraph. When $d=2$, we have 
		\begin{equation}
			\frac{\mathbb{H}_{\widetilde{A}}(\bm{0})}{\mathbb{H}_{A}(\bm{0})} \le \max_{v_1\in \widehat{F}_1,v_2\in \widecheck{F}_2} \frac{G_{\widetilde{A}}(v_1,v_2)}{G_{A}(v_1,v_2)}. 
		\end{equation}
	\end{lemma}
	\begin{proof}
		Combining the facts that $\mathbb{P}_x\left(\tau_{\widetilde{A}}=\tau_{\bm{0}} \right)=\mathbb{P}_x\left(\tau_{A}=\tau_{\bm{0}} \right)+\mathbb{P}_x\left(\tau_{D}<\tau_{\widetilde{A}}= \tau_{\bm{0}} \right)$ and $\mathbb{P}_x\left(\tau_{F_1}<\tau_{A}= \tau_{\bm{0}} \right)\le \mathbb{P}_x\left(\tau_{A}=\tau_{\bm{0}} \right)$, we have 
		\begin{equation}\label{3.4}
			\frac{\mathbb{P}_x\left(\tau_{\widetilde{A}}=\tau_{\bm{0}} \right)}{\mathbb{P}_x\left(\tau_{A}=\tau_{\bm{0}} \right)} \le  1+  \frac{\mathbb{P}_x\left( \tau_D<\tau_{\widetilde{A}}=\tau_{\bm{0}}\right)}{\mathbb{P}_x\left(\tau_{F_1}<\tau_{A}= \tau_{\bm{0}} \right)}. 
		\end{equation}
		Putting (\ref{3.1}), (\ref{3.2}) and (\ref{3.4}) together, we obtain 
		\begin{equation*}
			\begin{split}
				\frac{\mathbb{P}_x\left(\tau_{\widetilde{A}}=\tau_{\bm{0}} \right)}{\mathbb{P}_x\left(\tau_{A}=\tau_{\bm{0}} \right)} \le &1+ \max_{v_1\in \widehat{F}_1,v_2\in \widecheck{F}_2} \frac{\sum_{w\in D} \mathbb{P}_{v_1} \left(\tau_A=\tau_w \right)  G_{\widetilde{A}}(w,v_2)}{G_{A}(v_1,v_2)}\\
				=& \max_{v_1\in \widehat{F}_1,v_2\in \widecheck{F}_2} \frac{G_{\widetilde{A}}(v_1,v_2)}{G_{A}(v_1,v_2)}\ \ \ \ \ (\text{by}\ (\ref{2.5})). 
			\end{split}	
		\end{equation*}
		By taking the limit as $|x|\to \infty$, we conclude the desired bound.
	\end{proof}

	Next, we provide some technical results to facilitate the estimation of $\frac{G_{\widetilde{A}}(v_1,v_2)}{G_{A}(v_1,v_2)}$. For a given $v_2$, let $\mathring{v}_1$ be the vertex in $\widehat{F}_1\cup \widecheck{F}_1$ that maximizes $G_{\widetilde{A}}(\cdot, v_2)$ (if there is more than one choice, we select one of them in a predetermined manner). 
	\begin{lemma}\label{lemma_tilde_v1}
		For any $v_1\in \widehat{F}_1$ and $v_2\in \widecheck{F}_2$, we have 
		\begin{equation}\label{tilde_v1_e1}
			G_{\widetilde{A}}(v_1,v_2) \le \left[ 1-\mathbb{P}_{\mathring{v}_1}\left(\tau_D<\tau_{\widetilde{A}} \right) \right]^{-1} G_A(\mathring{v}_1,v_2). 
		\end{equation}
		Moreover, 
		\begin{equation}\label{tilde_v1_e2}
				G_{\widetilde{A}}(v_1,v_2) \le \left[ 1-\mathbb{P}_{\mathring{v}_1}\left(\tau_D<\tau_{\widetilde{A}} \right) \right]^{-1}\left[\mathbb{P}_{v_1}\left(\tau_{v_2}<\tau_{A} \right)  \right] ^{-1} G_A(v_1,v_2). 
		\end{equation}
	\end{lemma}
	\begin{proof}
	By applying (\ref{2.5}) with $A'=\widetilde{A}$, $x=\mathring{v}_1$ and $y=v_2$, we have  
	\begin{equation}\label{new_3.6}
		\begin{split}
			G_{\widetilde{A}}(\mathring{v}_1,v_2) =  G_{A}(\mathring{v}_1, v_2) + \sum_{w\in D}\mathbb{P}_{\mathring{v}_1} \left(\tau_{D}=\tau_{w}< \tau_{\widetilde{A}} \right) G_{\widetilde{A}}(w, v_2).  
		\end{split}	
	\end{equation}
	We claim that any path $\eta$ from $w\in D$ to $v_2\in \widecheck{F}_2$ without hitting $\widetilde{A}$ must intersect $\widecheck{F}_1$. In what follows, we prove this claim separately in two cases: 
	\begin{enumerate}
		\item When $v_2\in F_1$, since $F_1\subset F_2$, one has $(A\cup F_2)^{c}_{\bm{0}}\subset(A\cup F_1)^{c}_{\bm{0}}$ and therefore, $v_2\in \widecheck{F}_2\cap F_1\subset \widecheck{F}_1$, which indicates that $\eta$ intersects $\widecheck{F}_1$.

		\item When $v_2\in F_2\setminus F_1$, we decompose $\eta$ into two sub-paths $\eta_1:=(\eta(0),...,\eta(\ell))$ and $\eta_2:=(\eta(\ell+1),...,\eta(\boldsymbol{\mathrm{L}}(\eta)))$, where $\ell$ is the last time when $\eta$ is in $A\cup F_1$. Since $\boldsymbol{\mathrm{R}}(\eta_2)\subset (A\cup F_1)^c$ and $\eta_2(-1)=v_2\in \widecheck{F}_2\subset (A\cup F_1)^c_{\bm{0}}$, we have $\eta(\ell+1)\in (A\cup F_1)^c_{\bm{0}}$ and therefore, $\eta(\ell)\in \widecheck{F}_1$. Thus, $\eta$ intersects $\widecheck{F}_1$.

	\end{enumerate}
	With the claim above confirmed and by strong Markov property, we have
		\begin{equation*}
		G_{\widetilde{A}}(w, v_2) = \sum_{w'\in \widecheck{F}_1}\mathbb{P}_w\left(\tau_{\widecheck{F}_1}=\tau_{w'}<\infty\right) G_{\widetilde{A}}(w', v_2).  
	\end{equation*}
	Therefore, by the maximality of $G_{\widetilde{A}}(\mathring{v}_1, v_2)$, we get $$G_{\widetilde{A}}(w, v_2)\le \max_{w'\in \widecheck{F}_1} G_{\widetilde{A}}(w', v_2) \le G_{\widetilde{A}}(\mathring{v}_1, v_2), \ \forall w\in D.$$
	Combined with (\ref{new_3.6}), it yields that 
	\begin{equation}\label{new_3.7}
		G_{\widetilde{A}}(\mathring{v}_1,v_2) \le \left[1-\mathbb{P}_{\mathring{v}_1}\left(\tau_D<\tau_{\widetilde{A}} \right) \right]^{-1} G_A(\mathring{v}_1,v_2).  
	\end{equation}
By (\ref{new_3.7}) and the fact that $G_{\widetilde{A}}(v_1,v_2) \le G_{\widetilde{A}}(\mathring{v}_1,v_2)$, we conclude (\ref{tilde_v1_e1}).

It follows from (\ref{2.4}) that $$G_A(\mathring{v}_1,v_2)\le G_A(v_2,v_2)=\left[\mathbb{P}_{v_1}\left(\tau_{v_2}<\tau_{A} \right)  \right] ^{-1} G_A(v_1,v_2).$$ 
Combined with (\ref{tilde_v1_e1}), it implies (\ref{tilde_v1_e2}) and completes the proof.
	\end{proof}

	\subsection{Proof of Theorem \ref{thm1} for $\mathbb{Z}^2$}\label{subsection_proof_psi_finite}
	
	Arbitrarily take $A\in \mathcal{A}(\mathbb{Z}^2)$. To conclude Theorem \ref{thm1}, it suffices to show that we can select a vertex $z_\dagger\in A$ such that $\rho_A(z_{\dagger})\le C$ (for some prefixed constant $C>1$ independent of $A$). If there exists $z'\in A\setminus \{\bm{0}\}$ such that $N^A_{\infty}(z')=\emptyset$, then one has $\rho_A(z')=1$ (since $\mathbb{H}_A(z')=0$) and therefore, we can take $z_{\dagger}=z'$. Moreover, $N^A_{\infty}(\bm{0})\neq\emptyset$ since $A\in \mathcal{A}(\mathbb{Z}^2)$. Thus, without loss of generality, we thereafter assume that $N^A_{\infty}(z)\neq \emptyset$ for all $z\in A$. Here is our strategy for finding the appropriate vertex $z_\dagger$ to remove: 
	\begin{enumerate}
		\item When $A$ is $*$-connected, we first define $z''$ as an arbitrary vertex in $N^A_{\infty}(\bm{0})$. Let $\eta_{\diamond}$ be the path that starts from $z''$, moves along $N_*(\bm{0})$ in the clockwise direction and ends when it intersects $A$ (this intersection must happen since $N_*(\bm{0})\cap A\neq \emptyset$, which is ensured by $|A|\ge 2$ and the $*$-connectivity of $A$). Then we take $z_{\diamond}:= \eta_{\diamond}(-1)$. There are two subcases as follows.

		\begin{enumerate}
			\item    If $z_\diamond$ is not a $*$-cut vertex of $A$ (and therefore is a marginal vertex of $A$ by Lemma \ref{lemma_cut_marginal}), then we take $z_{\dagger}=z_{\diamond}$.

			\item    Otherwise (i.e. $z_\diamond$ is a $*$-cut vertex), we denote by $A_{\diamond}$ the $*$-cluster of $A\setminus \{z_{\diamond}\}$ that contains $\bm{0}$. Then we define $z_\dagger$ as an arbitrary marginal vertex of $A$ in $A\setminus A_{\diamond}$ (the existence of $z_\dagger$ is ensured by Item (2) of Lemma \ref{lemma_marginal}).

		\end{enumerate}

		\item When $A$ is not $*$-connected and there exists some $*$-cluster $A'$ of $A\setminus A_{(\bm{0})}$ with $|A'|\ge 2$ (recall that $A_{(\bm{0})}$ is the $*$-cluster of $A$ that contains $\bm{0}$), we choose $z_\dagger$ as an arbitrary marginal vertex of $A'$.

		\item When $A$ is not $*$-connected and every $*$-cluster of $A \setminus A_{(\bm{0})}$ contains only one vertex, let $z_{\dagger}$ be an arbitrary vertex in $A\setminus A_{(\bm{0})}$.

	\end{enumerate}

	Before bounding $\rho_A(z_{\dagger})$, let us give some observations for the selection of $z_\dagger$ in different cases.

	In Case (1a), let $z_{\ddagger}:=\eta_{\diamond}(\boldsymbol{\mathrm{L}}(\eta_{\diamond})-1)$ be the second last vertex of $\eta_{\diamond}$. Note that $z_{\ddagger}\in N^A_{\infty}(z_\dagger)$. Moreover, when $z_{\ddagger}\in N(\bm{0})$, one has $\mathbb{P}_{z_{\ddagger}}(\tau_A=\tau_{\bm{0}})\ge \tfrac{1}{4}$; when $z_{\ddagger}\in N_*(\bm{0})\setminus N(\bm{0})$, the random walk from $z_{\ddagger}$ may hit $A$ at $\bm{0}$ in two steps (since it can move counter-clockwise along $\eta_{\diamond}$ to reach $N(\bm{0})$ and then get to $\bm{0}$) and therefore, $\mathbb{P}_{z_{\ddagger}}(\tau_A=\tau_{\bm{0}})\ge \tfrac{1}{16}$. Thus, there exists $z_{\ddagger}\in N^A_{\infty}(z_\dagger)$ such that $\mathbb{P}_{z_{\ddagger}}(\tau_A=\tau_{\bm{0}})\ge \tfrac{1}{16}$.

	In Case (1b), we denote by $A_{\dagger}$ the $*$-cluster of $A\setminus \{z_{\diamond}\}$ that contains $z_\dagger$. We claim that for any finite cluster $D$ of $A^c$ with $\bm{0}\in \partial^{\mathrm{o}}_\infty D$, one has $A_{\dagger}\cap \partial^{\mathrm{o}}_\infty D=\emptyset$. Note that $\partial^{\mathrm{o}}_\infty D\subset A$. In fact, $D$ satisfies all conditions in Lemma \ref{lemma_connectivity_new} (recalling the assumption that $N^A_{\infty}(z)\neq \emptyset$ for all $z\in A$). Therefore, by Lemmas \ref{lemma_connectivity} and \ref{lemma_connectivity_new}, $(\partial^{\mathrm{o}}_\infty D)\setminus \{z_\diamond\}$ is $*$-connected (note that $z_\diamond$ is not necessarily in $\partial^{\mathrm{o}}_\infty D$). Thus, by the maximality of $*$-cluster, it follows that either $[(\partial^{\mathrm{o}}_\infty D)\setminus \{z_\diamond\}] \cap A_\dagger=\emptyset$ or $(\partial^{\mathrm{o}}_\infty D)\setminus \{z_\diamond\}\subset A_\dagger$, where the latter scenario is false since $\bm{0}\in (\partial^{\mathrm{o}}_\infty D)\setminus \{z_\diamond\}$ and $\bm{0}\notin A_\dagger$. With the claim above proved, one may have that any path from $N^{A}(z_\dagger)$ to $N^{A}(\bm{0})$ without intersecting $A$ must be contained in $A^c_\infty$. I.e., $N^{A}_{\bm{0}}(z_\dagger)\subset N^{A}_\infty(z_\dagger)$. Moreover, since $A\in \mathcal{A}(\mathbb{Z}^d)$, one has $A_\infty^c\subset A_{\bm{0}}^c$ and therefore, $N^{A}_{\bm{0}}(z_\dagger)  \supset N^{A}_\infty(z_\dagger)$. Thus, we obtain $N^{A}_{\bm{0}}(z_\dagger)=N^{A}_\infty(z_\dagger)$.

	In Case (2), since the $*$-cluster $A'$ (which contains $z_\dagger$) satisfies $\partial_{\infty}^{\mathrm{o}}A'\neq \emptyset$ and is not $*$-connected to $A_{(\bm{0})}$, one has $A'\subset [A_{(\bm{0})}]_{\infty}^c$. For the same reason as in Case (1b), any path from $N^{A}(z_\dagger)$ to $N(\bm{0})$ without intersecting $A$ must begin in $N^{A}_\infty(z_\dagger)$, and be contained in $A_\infty^c$. Thus, we have  $N^{A}_\infty(z_\dagger)=N^{A}_{\bm{0}}(z_\dagger)$.

	To sum up, the selection of $z_\dagger$ can be categorized into the following types: 
	\begin{enumerate}[(i)]
		
		\item $z_\dagger$ is a marginal vertex of $A$ with $N_*(z_\dagger)\cap A\neq \emptyset$ such that there exists $z_\ddagger\in N^A_{\infty}(z_\dagger)$ satisfying $\mathbb{P}_{z_{\ddagger}}\left(\tau_A=\tau_{\bm{0}} \right)\ge \tfrac{1}{16}$ (included in Case (1a));

		\item $z_\dagger$ is a marginal vertex of some cluster of $A$ such that $N_*(z_\dagger)\cap A\neq \emptyset$ and $N_\infty^A(z_\dagger)=N_{\bm{0}}^A(z_\dagger)$ (included in Cases (1b) and (2));

		\item $z_\dagger$ is a $*$-isolated vertex in $A$. I.e. $N_*(z_\dagger)\cap A=\emptyset$ (included in Case (3)).
		
	\end{enumerate}

	\begin{remark}
		At first glance, we may choose $z_{\dagger}$ as an arbitrary marginal vertex, without the need to separately consider Type (i) and Type (ii). However, the price of removing an arbitrary marginal vertex may still be unbounded. In the example shown by Figure \ref{fig:marginal}, recall that $z_3$ and $z_4$ are both marginal vertices of $A_{(\bm{0})}$ (i.e. the $*$-cluster of $A$ that contains $\bm{0}$). However, removing $z_4$ from $A$ will create a short-cut for the random walk to hit $\bm{0}$, thereby causing a high price. Our strategy suggests removing $z_3$ instead, which costs only a constant price.

	\end{remark}

	In the remaining part of this section, we establish $\rho_A(z_{\dagger})\le C$ separately for different types of $z_\dagger$. 
	
	\textbf{Price of removing a vertex $z_\dagger$ of Type $\mathrm{(i)}$:} We take a sufficiently large $R$ such that $A\subset \Lambda(R)$ and then let $x$ be an arbitrary vertex in $[\Lambda(R)]^c$. By strong Markov property and the fact that $x\in A^c_{\infty}$, we have 
	\begin{equation}\label{3.5}
		\mathbb{P}_x\left( \tau_{z_\dagger} <  \tau_{A\setminus \{z_\dagger\}}=\tau_{\bm{0}}  \right)  \le   \sum_{v\in N_\infty^A(z_\dagger) }	\mathbb{P}_x\left( \tau_{A\cup N_\infty^A(z_\dagger) } = \tau_v \right).
	\end{equation}
	For each $v\in N_\infty^A(z_\dagger)$, since $z_\dagger$ is marginal (i.e. $N^A_\infty(z_\dagger)$ is connected in $N^A_{\infty,*}(z_\dagger)$), the random walk may move along $N_{\infty,*}^A(z_\dagger)$ from $v$ to $z_\ddagger$ within six steps (recalling $z_\ddagger$ in the definition of Type (i) $z_\dagger$), which implies that 
	\begin{equation}\label{3.6}
	\mathbb{P}_v(\tau_{z_\ddagger}<\tau_A)\ge \frac{1}{4^{6}},\ \forall v\in N_\infty^A(z_\dagger). 
	\end{equation}
	Combined with the fact that $\mathbb{P}_{z_{\ddagger}}\left(\tau_A=\tau_{\bm{0}} \right)\ge \tfrac{1}{16}$, it yields that 
		\begin{equation*}
		\begin{split}
			\mathbb{P}_x\left( \tau_{A}=\tau_{\bm{0}}\right) \ge& \mathbb{P}_x\left( \tau_{A\cup N_\infty^A(z_\dagger) } = \tau_v \right)  \mathbb{P}_v\left(\tau_{z_\ddagger}<\tau_A \right) \mathbb{P}_{z_\ddagger}\left(\tau_A=\tau_{\bm{0}} \right)\\
			\ge &\frac{1}{4^8}\cdot  \mathbb{P}_x\left( \tau_{A\cup N_\infty^A(z_\dagger) } = \tau_v \right) .
		\end{split}
	\end{equation*}
	Summing over $v\in N^A_\infty(z_\dagger)$, by (\ref{3.5}) and $|N^A_\infty(z_\dagger)|\le 4$ we get that
	\begin{equation}\label{3.7}
		\mathbb{P}_x\left( \tau_{z_\dagger} <  \tau_{A\setminus \{z_\dagger\}}=\tau_{\bm{0}}  \right) \le 4^9 \cdot  \mathbb{P}_x\left( \tau_{A}=\tau_{\bm{0}}\right). 
	\end{equation}
	By (\ref{3.7}) and $\mathbb{P}_x( \tau_{A\setminus \{z_\dagger\}}=\tau_{\bm{0}})=\mathbb{P}_x( \tau_{A}=\tau_{\bm{0}})+\mathbb{P}_x( \tau_{z_\dagger} <  \tau_{A\setminus \{z_\dagger\}}=\tau_{\bm{0}})$, we have that 
	\begin{equation*}
		\mathbb{P}_x\left(  \tau_{A\setminus \{z_\dagger\}}=\tau_{\bm{0}}\right) \le (4^9+1)\cdot \mathbb{P}_x\left( \tau_{A}=\tau_{\bm{0}}\right). 
	\end{equation*}
	Letting $|x|\to \infty$, we obtain $\rho_A(z_\dagger)\le 4^9+1$.

	Now we consider the cases where the vertex $z_\dagger$ is of Type (ii) or Type (iii). To conclude $\rho_A(z_\dagger)\le C$, by Lemma \ref{lemma_upper_bound} (with $D=\{z_\dagger\}$ and $\widetilde{A}=A\setminus \{z_\dagger\} $), it suffices to prove that
	\begin{equation}\label{3.8}
		G_{A\setminus \{z_\dagger\}}(v_1,v_2)\le C\cdot G_{A}(v_1,v_2),\ \forall v_1\in \widehat{F}_1\ \text{and} \ v_2\in \widecheck{F}_2, 
	\end{equation}
	where $F_1,F_2\subset \mathbb{Z}^2$ are some finite sets with $F_j\cap A=\{z_\dagger\}$ for $j\in \{1,2\}$ (recalling that $\widehat{F}_j:=[\partial^{\mathrm{i}}_{\infty} (A\cup F_j)]\setminus \widetilde{A}$ and $\widecheck{F}_j:=[\partial^{\mathrm{i}}_{\bm{0}} (A\cup F_j)]\setminus \widetilde{A}$ ). We will determine $F_1$ and $F_2$ according to the type of $z_\dagger$. Recall that $\mathring{v}_1 $ maximizes $G_{\widetilde{A}}(\cdot, v_2)$ in $\widehat{F}_1\cup \widecheck{F}_j$.

	\textbf{Price of removing a vertex $z_\dagger$ of Type $\mathrm{(ii)}$:} In this case, we set $F_1=F_2=\{z_\dagger\}\cup N^A(z_\dagger)$. For $j\in \{1,2\}$, it is straightforward
	 that $\widehat{F}_j \subset N^A_\infty(z_\dagger)$. Moreover, we have $\widecheck{F}_j\subset N_{\bm{0}}^A(z_\dagger)=N_{\infty}^A(z_\dagger)$ (by the definition of Type (ii) $z_\dagger$). Therefore, $v_1,\mathring{v}_1,v_2\in N^A_\infty(z_\dagger)$. Since $N_*(z_\dagger)\cap A\neq \emptyset$, the random walk starting from $\mathring{v}_1\in N_\infty^A(z_\dagger)$ may reach $A\setminus \{z_\dagger\}$ within four steps, which implies that $1-\mathbb{P}_{\mathring{v}_1}(\tau_{z_\dagger}<\tau_{A\setminus \{z_\dagger\}} )=\mathbb{P}_{\mathring{v}_1}(\tau_{A\setminus \{z_\dagger\}}<\tau_{z_\dagger} )\ge \tfrac{1}{4^4}$ (note that for the equality, we used the recurrence of the random walk on $\mathbb{Z}^2$). In addition, for the same reason as proving (\ref{3.6}), we have $\mathbb{P}_{v_1}\left(\tau_{v_2}<\tau_{A} \right)\ge \frac{1}{4^6}$. Thus, by (\ref{tilde_v1_e2}), we conclude (\ref{3.8}) with $C=4^{10}$.

	\textbf{Price of removing a vertex $z_\dagger$ of Type $\mathrm{(iii)}$:} Let $R:=\boldsymbol{\mathrm{D}}(z_\dagger,A\setminus \{z_\dagger\})$. When $1<R\le  C_6:=\lceil\max\{100,C_3(\frac{1}{4},\frac{1}{2},1,2 ),C_4\}\rceil$ (recall $C_3$ and $C_4$ in Lemmas \ref{lemma_compare_green} and \ref{lemma_hit_distant} respectively), let $F_1=F_2=\{z_\dagger\}\cup N(z_\dagger)$. It is easy to see that $\widehat{F}_j,\widecheck{F}_j\subset N(z_\dagger)$ for all $j\in \{1,2\}$. For any $v_1\in\widehat{F}_1\subset N(z_\dagger),v_2\in \widecheck{F}_2\subset N(z_\dagger)$, the random walk starting from $v_1$ can get $v_2$ along $N_*(z_\dagger)$ within $4$ steps (recall that $N_*(z_\dagger)\cap A=\emptyset$). Moreover, for any $\mathring{v}_1\in \widehat{F}_1\cup \widecheck{F}_1\subset N(z_\dagger)$, since $R\le C_6$, we know that the random walk starting from $\mathring{v}_1$ may reach $A\setminus \{z_{\dagger}\}$ within $C_6+3$ steps without hitting $z_\dagger$ (note that there exists some vertex $w\in N(z_{\dagger})$ such that $\mathbf{\boldsymbol{\mathrm{D}}}(w,A\setminus \{z_\dagger\})=R-1\le C_6-1$, and that the random walk may move from $\mathring{v}_1$ to $w$ along $N_*(z_\dagger)$ within $4$ steps). These two facts indicate $\mathbb{P}_{v_1}\left(\tau_{v_2}<\tau_{A} \right)\ge \frac{1}{4^4}$ and $\mathbb{P}_{\mathring{v}_1}(\tau_{A\setminus \{z_\dagger\}}<\tau_{z_\dagger} )\ge  \tfrac{1}{4^{C_6+3}}$ respectively. Thus, by (\ref{tilde_v1_e2}), we conclude (\ref{3.8}) with $C=4^{C_6+7}$.

	Now we focus on the case when $R> C_6$. We set $F_1=\mathbf{B}_{z_\dagger}(\frac{R}{4})$ and $F_2=\mathbf{B}_{z_\dagger}(\frac{R}{2})$. It follows that $\widehat{F}_1=\widecheck{F}_1=\partial^{\mathrm{i}}\mathbf{B}_{z_\dagger}(\frac{R}{4})$ and $\widehat{F}_2=\widecheck{F}_2=\partial^{\mathrm{i}}\mathbf{B}_{z_\dagger}(\frac{R}{2})$. Hence, one has $v_1,\mathring{v}_1\in \partial^{\mathrm{i}}\mathbf{B}_{z_\dagger}(\frac{R}{4})$ and $v_2\in \partial^{\mathrm{i}}\mathbf{B}_{z_\dagger}(\frac{R}{2})$. By (\ref{tilde_v1_e1}), we have 
	\begin{equation}\label{new_3.13}
		G_{A\setminus \{z_{\dagger}\}}(v_1, v_2)\le \left[\mathbb{P}_{\mathring{v}_1} \left( \tau_{A\setminus \{z_\dagger\}}<\tau_{z_\dagger} \right) \right]^{-1} G_{A}(\mathring{v}_1, v_2). 
	\end{equation}
	We arbitrarily take $z_{\#}\in A\cap \partial^{\mathrm{i}} \mathbf{B}_{z_\dagger}(R)$. By Lemma \ref{lemma_hit_distant}, we have 
	\begin{equation}\label{new_3.14}
		\mathbb{P}_{\mathring{v}_1} \left( \tau_{A\setminus \{z_\dagger\}}<\tau_{z_\dagger} \right) \ge \mathbb{P}_{\mathring{v}_1} \left( \tau_{z_{\#} }<\tau_{z_\dagger} \right) \ge c_4. 
	\end{equation}
	Moreover, by Lemma \ref{lemma_compare_green}, we also have 
	\begin{equation}\label{new_3.15}
	G_{A}(v_1, v_2)\ge c_3(\tfrac{1}{4},\tfrac{1}{2},1,2 )\cdot 	G_{A}(\mathring{v}_1, v_2).
	\end{equation}
	Combining (\ref{new_3.13}), (\ref{new_3.14}) and (\ref{new_3.15}), we conclude (\ref{3.8}) with $C=c_3^{-1}c_4^{-1}$ and finally complete the proof of $\psi(\mathbb{Z}^2)<\infty$.      \qed

	\begin{remark}[Generalization of Theorem \ref{thm1} for $\mathbb{Z}^2$] \label{generalize_2d}
		
		In the proof of $\psi(\mathbb{Z}^2)<\infty$ presented in this section, the existence of marginal vertices (Lemma \ref{lemma_marginal}) only relies on the planality of the graph. Moreover, for the estimates for the price of removing a vertex $z_{\dagger}$ of Type $\mathrm{(i)}$ and Type $\mathrm{(ii)}$, what essentially matters is that the probability for the random walk to surround each face is bounded away from $0$, which can be easily ensured when $\sup_{v}\mathrm{deg}(v)$ and $\sup_{\mathcal{S}}|\boldsymbol{\mathrm{e}}(\mathcal{S})|$ are both finite. Additionally, to bound the price of removing a vertex $z_{\dagger}$ of Type $\mathrm{(iii)}$, the essntial properties we rely on are summarized in Lemmas \ref{lemma_compare_green} and \ref{lemma_hit_distant}, whose proofs are based on the invariance principle. For some sufficient conditions of these properties, readers may refer to \cite[Sections 4.3 and 7.2]{barlow2017random}. To summarize, we expect that the vertex-removal stability holds for all planar graphs of bounded degree and bounded number of edges surrounding every face, where the random walk from each vertex converges to a Brownian motion under a uniform rate.  
		
	\end{remark}

	\section{Absence of vertex-removal stability for $\mathbb{Z}^d$ ($d\ge 3$)}\label{section_psi_infinite}

		This section includes the proof of $\psi(\mathbb{Z}^d)=\infty$ for $d\ge 3$. The key to achieve this is the following proposition: 
		
		\begin{proposition}\label{prop_1}
			In the case of $\mathbb{Z}^d$ for $d\ge 3$, there exists $c_6(d)>0$ and a sequence of sets $\{K_n\}_{n\ge 1}$ with $K_n\in \mathcal{A}(\mathbb{Z}^d)$ and $|K_n|=k_n$ (where $\{k_n\}_{n\ge 1}$ is an increasing sequence of integers) such that for any large enough $n$ and any $z\in K_n\setminus \{\bm{0}\}$, 
			\begin{equation}\label{4.1}
				\frac{\mathrm{Es}_{K_n\setminus \{z\}}(\bm{0})}{\mathrm{Es}_{K_n}(\bm{0})}\ge e^{c_6n}. 
			\end{equation}
		\end{proposition}

		In fact, Theorem \ref{thm1} for $d\ge 3$ follows immediately once Proposition \ref{prop_1} is proved:
		\begin{proof}[Proof of Theorem \ref{thm1} for $d\ge 3$ assuming Proposition \ref{prop_1}]
			Let $\{K_n\}_{n\ge 1}$ be the sequence presented by Proposition \ref{prop_1}. For any large enough $n$ and $z\in K_n\setminus \{\bm{0}\}$, by (\ref{2.14}), (\ref{2.16}) and Proposition \ref{prop_1}, we get 
			\begin{equation*}
				\frac{\mathbb{H}_{K_n\setminus \{z\}}(\bm{0})}{\mathbb{H}_{K_n}(\bm{0})}= 	\frac{\mathrm{Es}_{K_n\setminus \{z\}}(\bm{0})}{\mathrm{Es}_{K_n}(\bm{0})}\cdot \frac{	\mathrm{cap}(K_n)}{\mathrm{cap}(K_n\setminus \{z\})} \ge e^{c_6n}. 
			\end{equation*}
			Thus, it follows from Definition \ref{def1} that 
			\begin{equation}
				\psi(\mathbb{Z}^d)\ge \min_{z\in K_n\setminus \{\bm{0}\}}\frac{\mathbb{H}_{K_n\setminus \{z\}}(\bm{0})}{\mathbb{H}_{K_n}(\bm{0})} \ge e^{c_6n}.  
			\end{equation}
	 Since $n$ can be arbitrarily large, we conclude $\psi(\mathbb{Z}^d)=\infty$ for $d\ge 3$. 
		\end{proof}

		Subsequently, we present the construction of the required sequence of discrete Klein Bottles $\{K_n\}_{n\ge 1}$ (refer to Figure \ref{fig:Klein}), and then proceed to prove that they satisfy (\ref{4.1}). Roughly speaking, each $K_n$ is composed of three parts: the shell of a box, an outer tube and an inner tube. Specifically, $K_n$ can be constructed in the following steps:
		\begin{enumerate}
			\item[Step $1$:] Let $K_n^{1}:= \{x\in \mathbb{Z}^d: \text{exactly one coordinate of}\ x\ \text{equals to}\ n\}$ be the shell of the box $\Lambda(n)$, excluding corner vertices.

			\item[Step $2$:] Let $x_n^1:= (0,n)\times \{0\}^{d-2}$, $x_n^2:=(n)\times \{0\}^{d-1}$ and $x_n^3:=(\tfrac{n}{2},-n)\times \{0\}^{d-2}$. Then we remove these vertices and their neighbors from $K_n^{1}$ to form $$K_n^{2}:=K_n^{1}\setminus \cup_{j=1}^{3} N(x_n^j).$$

			\item[Step $3$:] We take the path $\eta_n^{\mathrm{o}}$ as the broken line joining the vertices $x_n^1$, $(0,2n)\times \{0\}^{d-2}$, $(2n,2n)\times \{0\}^{d-2}$, $(2n)\times \{0\}^{d-1}$ and $x_n^2$ in turn (as shown in Figure \ref{fig:Klein}). Then we define the outer tube 
			$$T_n^{\mathrm{o}}:=\big[\partial^{\mathrm{o}} \boldsymbol{\mathrm{R}}(\eta_n^{\mathrm{o}})\big]\setminus \Lambda(n). $$
			We also take the path $\eta_n^{\mathrm{i}}$ as the broken line joining $x_n^2$, $(\frac{n}{2})\times \{0\}^{d-1}$ and $x_n^3$ in turn (see also Figure \ref{fig:Klein}). We define the inner tube as 
			$$
			T_n^{\mathrm{i}}:=\big[\partial^{\mathrm{o}} \boldsymbol{\mathrm{R}}(\eta_n^{\mathrm{i}})\big]\cap  \Lambda(n-1). 
			$$
			We attach $\bm{0}$ and these two tubes to $K_n^2$, resulting in our target subset $$K_n:=K_n^2\cup T_n^{\mathrm{o}}\cup T_n^{\mathrm{i}}\cup \{\bm{0}\}.$$

		\end{enumerate}

	\begin{remark}\label{remark_observation}

		Here are some crucial observations for $K_n$:  
		\begin{enumerate}
			
			\item As shown by the pink trajectory in Figure \ref{fig:Klein}, if a random walk starts from $\bm{0}$ and escapes to infinity without hitting $K_n$, then it must first get to $x_n^1$, then reach $x_n^2$, then arrive at $x_n^3$, and finally escapes to infinity (during this process, it may move inside $K_n^1$, or go forward and backward inside the tubes, but cannot hit $K_n$).

			\item  For any $z\in K_n^2\cup T_n^{\mathrm{o}}\cup \{x_n^3\}$, a quick observation reveals that within $0.1n$ steps and without hitting $K_n\setminus \{z\}$, the random walk may move from $z$ to a position outside $\Lambda(n)$, whose graph distance to $K_n$ is at least $\frac{n}{100}$. I.e., there exists a path $\hat{\eta}_{z}^{\mathrm{o}}$ such that $\boldsymbol{\mathrm{L}}(\hat{\eta}_{z}^{\mathrm{o}})\le 0.1n$, $\boldsymbol{\mathrm{R}}(\hat{\eta}_{z}^{\mathrm{o}})\cap (K_n\setminus \{z\})=\emptyset$, $\hat{\eta}_{z}^{\mathrm{o}}(0)=z$, $\hat{\eta}_{z}^{\mathrm{o}}(-1)\in [\Lambda(n)]^c$ and $\boldsymbol{\mathrm{D}}(\hat{\eta}_{z}^{\mathrm{o}}(-1),K_n)\ge \frac{n}{100}$.

			Similarly, for any $z\in K_n^2\cup T_n^{\mathrm{i}}\cup \{x_n^1\}$,  there exists a path $\hat{\eta}_{z}^{\mathrm{i}}$ such that $\boldsymbol{\mathrm{L}}(\hat{\eta}_{z}^{\mathrm{i}})\le 0.1n$, $\boldsymbol{\mathrm{R}}(\hat{\eta}_{z}^{\mathrm{i}})\cap (K_n\setminus \{z\})=\emptyset$, $\hat{\eta}_{z}^{\mathrm{i}}(-1)=z$, $\hat{\eta}_{z}^{\mathrm{i}}(0)\in \Lambda(n-1)$ and $\boldsymbol{\mathrm{D}}(\hat{\eta}_{z}^{\mathrm{i}}(0),K_n)\ge \frac{n}{100}$.

			Furthermore, we also denote by $\hat{\eta}_{\bm{0}}$ the path of length $\lfloor 0.1n \rfloor$ such that $\hat{\eta}_{\bm{0}}(i)=(0,i)\times \{0\}^{d-2}$ for all $0\le i \le \boldsymbol{\mathrm{L}}(\hat{\eta}_{\bm{0}})$.

			\item  For any $z\in T_n^{\mathrm{o}}$, the random walk may start from $x_n^1$, then go forward along $\eta_n^{\mathrm{o}}$ to one of the neighbor of $z$, and finally reach $z$ by one step. I.e., there exists a path $\bar{\eta}_z^{\mathrm{o}}$ such that $\bar{\eta}_z^{\mathrm{o}}(0)=x_n^1$, $\bar{\eta}_z^{\mathrm{o}}(-1)=z$, $\boldsymbol{\mathrm{R}}(\bar{\eta}_z^{\mathrm{o}})\cap K_n=\{z\}$ and $\boldsymbol{\mathrm{L}}(\bar{\eta}_z^{\mathrm{o}})\le \boldsymbol{\mathrm{L}}(\eta_n^{\mathrm{o}})+1$.

			Likewise, for any $z\in T_n^{\mathrm{i}}$, there exists a path $\bar{\eta}_z^{\mathrm{i}}$ such that $\bar{\eta}_z^{\mathrm{i}}(0)=z$, $\bar{\eta}_z^{\mathrm{i}}(-1)=x_n^3$, $\boldsymbol{\mathrm{R}}(\bar{\eta}_z^{\mathrm{i}})\cap K_n=\{z\}$ and $\boldsymbol{\mathrm{L}}(\bar{\eta}_z^{\mathrm{i}})\le \boldsymbol{\mathrm{L}}(\eta_n^{\mathrm{i}})+1$.

			\item Arbitrarily removing a vertex $z$ from $K_n$ will create a new (macroscopically shorter) route for the random walk from $\bm{0}$ to escape to infinity without hitting $K_n$, which can be described separately by the position of $z$:

			\begin{enumerate}

				\item  When $z\in K_n^2$, the random walk starting from $\bm{0}$ and escaping to infinite without hitting $K_n$ may first move along $\hat{\eta}_{\bm{0}}$ (``moving along a path $\eta$'' means that the random walk starts from $\eta(0)$ and ends at $\eta(-1)$ and in addition, it can go forward or backward along $\eta$, but cannot intersect $\partial^{\mathrm{o}}\boldsymbol{\mathrm{R}}(\eta)$), then get to $\hat{\eta}_z^{\mathrm{i}}(0)$, next move along the path $\hat{\eta}_z^{\mathrm{i}}\circ \hat{\eta}_z^{\mathrm{o}}$, and finally escape to infinity from $\hat{\eta}_z^{\mathrm{o}}(-1)$.

				\item  When $z\in T_n^{\mathrm{o}}$, the aforementioned random walk may first move along $\hat{\eta}_{\bm{0}}$, then get $\hat{\eta}_{x_n^1}^{\mathrm{i}}(0)$, next move along $\hat{\eta}_{x_n^1}^{\mathrm{i}}\circ \bar{\eta}_z^{\mathrm{o}}\circ \hat{\eta}_{z}^{\mathrm{o}}$, and finally escape to infinity from $\hat{\eta}_z^{\mathrm{o}}(-1)$.

				\item  When $z\in T_n^{\mathrm{i}}$, it may first move along $\hat{\eta}_{\bm{0}}$, then get $\hat{\eta}_{z}^{\mathrm{i}}(0)$, next move along $\hat{\eta}_{z}^{\mathrm{i}}\circ \bar{\eta}_z^{\mathrm{i}}\circ \hat{\eta}_{x_n^3}^{\mathrm{o}}$, and finally escape to infinity from $\hat{\eta}_{x_n^3}^{\mathrm{o}}(-1)$.

			\end{enumerate}

		\end{enumerate}

\end{remark}

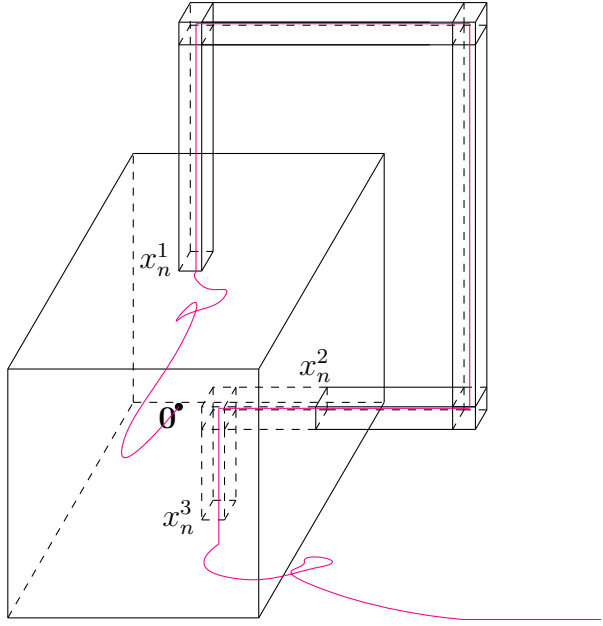
\begin{figure}[h]
	\begin{tikzpicture}[x={(1cm,0cm)}, y={(0cm,1cm)}, z={({cos(60)},{sin(60)})}, scale=0.3]
		
		x={(1cm,0cm)}; y={(0cm,1cm)}; z={({cos(60)},{sin(60)})};
		\def\cubeSize{11}
		\def\s{\cubeSize}
		
		
		\draw (0, 0, 0) -- (\cubeSize,0,0);
		\draw (0, 0, 0) -- (0,\cubeSize,0);
		\draw[dashed] (0, 0, 0) -- (0,0,\cubeSize);
		\draw (0, \cubeSize, 0) -- (0,\cubeSize,\cubeSize);
		\draw (0, \cubeSize, 0) -- (\cubeSize,\cubeSize,0);
		\draw (\cubeSize, \cubeSize, \cubeSize) -- (\cubeSize,\cubeSize,0);
		\draw (\cubeSize, \cubeSize, \cubeSize) -- (0,\cubeSize,\cubeSize);
		\draw (\cubeSize, \cubeSize, \cubeSize) -- (\cubeSize,0,\cubeSize);
		\draw (\cubeSize,0, 0) -- (\cubeSize,\cubeSize,0);
		\draw (\cubeSize,0, 0) -- (\cubeSize,0,\cubeSize);
		\draw[dashed]  (0, 0, \cubeSize) -- (\cubeSize,0,\cubeSize);   
		\draw[dashed]  (0, 0, \cubeSize) -- (0, \cubeSize,\cubeSize);          
		
		\pgfmathsetmacro{\mid}{(\cubeSize-1) / 2};
		\fill (\mid, \mid, \mid) circle (\mid pt);
		\node at (\mid-0.5, \mid-0.5, \mid) []{$\bm{0}$}; 
		\node at (\mid-1, \cubeSize+0.5, \mid) []{$x_n^1$}; 
		\node at (\cubeSize, \mid+1.8, \mid) []{$x_n^2$}; 
		\node at (0.5*\cubeSize-0.5, -0.08*\cubeSize+1, \mid) []{$x_n^3$};

		\draw (\mid, \cubeSize, \mid) -- (\mid+1, \cubeSize, \mid);

		\draw (\mid+1, \cubeSize, \mid) -- (\mid+1, \cubeSize, \mid+1);  
		
		\draw [dashed](\mid, \cubeSize, \mid) -- (\mid, \cubeSize, \mid+1);  
		
		\draw [dashed](\mid, \cubeSize, \mid+1) -- (\mid+1, \cubeSize, \mid+1);

		\draw (\mid, \cubeSize, \mid) -- (\mid, \cubeSize+\s, \mid);        
		\draw (\mid+1, \cubeSize, \mid) -- (\mid+1, \cubeSize+\s, \mid);  
		\draw (\mid+1, \cubeSize, \mid+1) -- (\mid+1, \cubeSize+\s, \mid+1);  
		\draw[dashed] (\mid, \cubeSize, \mid+1) -- (\mid, \cubeSize+\s, \mid+1);  
		
		\draw (\mid+1, \cubeSize+\s, \mid) -- (\mid+1+\s, \cubeSize+\s, \mid);  
		\draw (\mid, \cubeSize+\s, \mid) -- (\mid+\s, \cubeSize+\s, \mid);    
		\draw(\mid, \cubeSize+\s, \mid+1) -- (\mid, \cubeSize+\s, \mid+1);  
		
		\draw(\mid, \cubeSize+\s, \mid+1) -- (\mid+\s+1, \cubeSize+\s, \mid+1);  
		\draw[dashed](\mid, \cubeSize+\s-1, \mid+1) -- (\mid+1+\s, \cubeSize+\s-1, \mid+1);  
		
		
		\draw (\mid+1, \cubeSize+\s-1, \mid) -- (\mid+1+\s, \cubeSize+\s-1, \mid);  
		\draw (\mid, \cubeSize+\s-1, \mid) -- (\mid+\s, \cubeSize+\s-1, \mid);    
		\draw(\mid, \cubeSize+\s-1, \mid+1) -- (\mid, \cubeSize+\s-1, \mid+1);  
		
		\draw[dashed] (\mid+1, \cubeSize+\s-1, \mid) -- (\mid+1, \cubeSize+\s-1, \mid+1);
		\draw[dashed] (\mid, \cubeSize+\s-1, \mid) -- (\mid, \cubeSize+\s-1, \mid+1);
		\draw (\mid+1, \cubeSize+\s, \mid) -- (\mid+1, \cubeSize+\s, \mid+1);
		\draw (\mid, \cubeSize+\s, \mid) -- (\mid, \cubeSize+\s, \mid+1);

		\draw[dashed](\mid+1+\s, \cubeSize+\s-1, \mid+1) -- (\mid+1+\s, \cubeSize+\s-1, \mid+1-1);
		\draw(\mid+1+\s, \cubeSize+\s-1, \mid) -- (\mid+2+\s, \cubeSize+\s-1, \mid);
		\draw(\mid+1+\s, \cubeSize+\s-1, \mid) -- (\mid+1+\s, \cubeSize+\s, \mid);
		
		\draw(\mid+2+\s, \cubeSize+\s-1, \mid)-- (\mid+2+\s, \cubeSize+\s-1, \mid+1);
		\draw(\mid+2+\s, \cubeSize+\s-1, \mid+1)-- (\mid+2+\s, \cubeSize+\s, \mid+1);
		\draw[dashed](\mid+2+\s, \cubeSize+\s-1, \mid+1)-- (\mid+1+\s, \cubeSize+\s-1, \mid+1);
		\draw[dashed](\mid+1+\s, \cubeSize+\s-1, \mid+1)-- (\mid+1+\s, \cubeSize+\s, \mid+1);
		\draw(\mid+1+\s, \cubeSize+\s, \mid)-- (\mid+1+\s, \cubeSize+\s, \mid+1);
		\draw(\mid+1+\s, \cubeSize+\s, \mid+1)-- (\mid+2+\s, \cubeSize+\s, \mid+1);
		\draw(\mid+1+\s, \cubeSize+\s, \mid)-- (\mid+2+\s, \cubeSize+\s, \mid);
		\draw(\mid+2+\s, \cubeSize+\s, \mid)-- (\mid+2+\s, \cubeSize+\s-1, \mid);
		\draw(\mid+2+\s, \cubeSize+\s, \mid)-- (\mid+2+\s, \cubeSize+\s, \mid+1);

		\draw(\mid+1+\s, \cubeSize+\s-1, \mid)--(\mid+1+\s, \mid+2-2, \mid);
		\draw(\mid+2+\s, \cubeSize+\s-1, \mid)--(\mid+2+\s, \mid+2-2, \mid);
		\draw[dashed](\mid+1+\s, \cubeSize+\s-1, \mid+1)--(\mid+1+\s, \mid+2-2, \mid+1);
		\draw(\mid+2+\s, \cubeSize+\s-1, \mid+1)--(\mid+2+\s, \mid+2-2, \mid+1);

		\draw(\mid+1+\s, \mid+2-2, \mid)--(\mid+2+\s, \mid+2-2, \mid);
		\draw[dashed](\mid+1+\s, \mid+2-2, \mid)--(\mid+1+\s, \mid+2-2, \mid+1);
		\draw[dashed](\mid+1+\s, \mid+2-2, \mid+1)--(\mid+2+\s, \mid+2-2, \mid+1);
		\draw(\mid+2+\s, \mid+2-2, \mid+1)--(\mid+2+\s, \mid+2-2, \mid);
		\draw(\mid+2+\s, \mid+2-2, \mid)--(\mid+2+\s, \mid+2-3, \mid);
		\draw(\mid+1+\s, \mid+2-2, \mid)--(\mid+1+\s, \mid+2-3, \mid);
		\draw[dashed](\mid+1+\s, \mid+2-2, \mid+1)--(\mid+1+\s, \mid+2-3, \mid+1);
		\draw(\mid+2+\s, \mid+2-2, \mid+1)--(\mid+2+\s, \mid+2-3, \mid+1);
		\draw[dashed](\mid+2+\s, \mid+2-3, \mid+1)--(\mid+1+\s, \mid+2-3, \mid+1);
		\draw(\mid+2+\s, \mid+2-3, \mid+1)--(\mid+2+\s, \mid+2-3, \mid);
		\draw[dashed](\mid+1+\s, \mid+2-3, \mid+1)--(\mid+1+\s, \mid+2-3, \mid);
		
		\draw(\mid+2+\s, \mid+2-3, \mid)--(\s, \mid+2-3, \mid);
		
		\draw(\mid+2+\s, \mid+2-2, \mid)--(\s, \mid+2-2, \mid);
		\draw(\mid+2+\s-1, \mid+2-2, \mid+1)--(\s, \mid+2-2, \mid+1);
		\draw[dashed](\mid+1+\s, \mid+2-3, \mid+1)--(\s, \mid+2-3, \mid+1);
		\draw[dashed](\s, \mid+2-3, \mid)--(\s, \mid+2-3, \mid+1);
		\draw(\s, \mid+2-3, \mid)--(\s, \mid+2-2, \mid);
		\draw[dashed](\s, \mid+2-2, \mid+1)--(\s, \mid+2-3, \mid+1);
		\draw(\s, \mid+2-2, \mid+1)--(\s, \mid+2-2, \mid);
		
		\draw[dashed] (\s, \mid+2-3, \mid)--(\mid+1, \mid+2-3, \mid);
		\draw[dashed] (\s, \mid+2-3, \mid+1)--(\mid+1, \mid+2-3, \mid+1);
		\draw[dashed] (\s, \mid+2-2, \mid)--(\mid+1, \mid+2-2, \mid);
		\draw[dashed] (\s, \mid+2-2, \mid+1)--(\mid+1, \mid+2-2, \mid+1);

		\draw[dashed] (\mid+1, \mid+2-3, \mid)--(\mid+1, \mid+2-3, \mid+1);
		\draw[dashed] (\mid+1, \mid+2-3, \mid)--(\mid+1, \mid+2-2, \mid);
		\draw[dashed] (\mid+1, \mid+2-2, \mid+1)--(\mid+1, \mid+2-3, \mid+1);
		\draw[dashed] (\mid+1, \mid+2-2, \mid+1)--(\mid+1, \mid+2-2, \mid);
		
		\draw[dashed] (\mid+2, \mid+2-3, \mid)--(\mid+2, \mid+2-3, \mid+1);
		\draw[dashed] (\mid+2, \mid+2-3, \mid)--(\mid+2, \mid+2-2, \mid);
		\draw[dashed] (\mid+2, \mid+2-2, \mid+1)--(\mid+2, \mid+2-3, \mid+1);
		\draw[dashed] (\mid+2, \mid+2-2, \mid+1)--(\mid+2, \mid+2-2, \mid);
		
		\draw[dashed] (\mid+1, \mid+2-3, \mid)--(\mid+1, 0, \mid);
		\draw[dashed] (\mid+2, \mid+2-3, \mid)--(\mid+2, 0, \mid);
		\draw[dashed] (\mid+1, \mid+2-3, \mid+1)--(\mid+1, 0, \mid+1);
		\draw[dashed] (\mid+2, \mid+2-3, \mid+1)--(\mid+2, 0, \mid+1);
		
		\draw[dashed] (\mid+1, 0, \mid)--(\mid+2, 0, \mid);
		\draw[dashed] (\mid+1, 0, \mid)--(\mid+1, 0, \mid+1);
		\draw[dashed] (\mid+2, 0, \mid)--(\mid+2, 0, \mid+1);
		\draw[dashed] (\mid+1, 0, \mid+1)--(\mid+2, 0, \mid+1);
		
		\draw[magenta] (\mid+1.5, \mid-0.5, \mid+0.5)--(\mid+1.5,-1.5, \mid+0.5);
		\draw [magenta] plot [smooth, tension=1.5] coordinates { (\mid+1.5,-1.5, \mid+0.5)(\mid+2,-2.5, \mid) (\mid+6,-2.5, \mid+0.5) (\mid+5,-2.7, \mid+1) (\mid+6,-4, \mid+1)(\mid+11,-7, \mid+3)};
		
		\draw [magenta] plot [smooth, tension=1.5] coordinates {(\mid+11,-7, \mid+3)(\s+11,-7, \mid+3) };
		
		\draw[magenta] (\mid+1.5, \mid-0.5, \mid+0.5)--(\s+\mid+1.5,\mid-0.5, \mid+0.5);
		\draw[magenta] (\s+\mid+1.5,\mid-0.5, \mid+0.5)--(\s+\mid+1.5,2*\mid+\s+0.5, \mid+0.5);
		\draw[magenta] (\s+\mid+1.5,2*\mid+\s+0.5, \mid+0.5)--(\mid+0.5,2*\mid+\s+0.5, \mid+0.5);
		\draw[magenta] (\mid+0.5,2*\mid+\s+0.5, \mid+0.5)--(\mid+0.5,\s-0.5, \mid+0.5);
		\draw [magenta] plot [smooth, tension=1.5] coordinates { (\mid+0.5,\s-0.5, \mid+0.5) (\mid+0.5,\s-1.5, \mid+1)(\mid+1.5,\s-2, \mid+1)(\mid-0.5,\s-3, \mid+1)(\mid-1,\s-5, \mid+3)(\mid-2,\s-7, \mid-1)(\mid,\mid,\mid)};

	\end{tikzpicture}
	\caption{The discrete Klein bottle for $\psi(\mathbb{Z}^d)=\infty$ \label{fig:Klein}}
\end{figure}


Before proving Proposition \ref{prop_1}, we need the following formula for the probability that a random walk moves along a fixed path without hitting its boundary.

\begin{lemma}\label{lemma_move_along}
	For any self-avoiding path $\eta$ on $\mathbb{Z}^d$, let $\gamma(\eta):=\mathbb{P}_{\eta(0)} \left[  \tau_{\eta(-1)}< \tau_{\partial^{\mathrm{o}}\boldsymbol{\mathrm{R}}(\eta)}  \right]$. Then we have 
	\begin{equation}\label{new_4.3}
		\gamma(\eta):= \frac{2\sqrt{d^2-1}}{(d+\sqrt{d^2-1} )^{\boldsymbol{\mathrm{L}}(\eta)+1} -(d-\sqrt{d^2-1} )^{\boldsymbol{\mathrm{L}}(\eta)+1}  }.
	\end{equation}
\end{lemma}
\begin{proof}
	For any $0\le i\le \boldsymbol{\mathrm{L}}(\eta)$, let $q_i:=\mathbb{P}_{\eta(i)} \left[  \tau_{\eta(-1)}< \tau_{\partial^{\mathrm{o}}\boldsymbol{\mathrm{R}}(\eta)}  \right]$. Then the number sequence $\{q_i\}_{i=0}^{\boldsymbol{\mathrm{L}}(\eta)}$ satisfies the following:
	\begin{itemize}
		\item  $q_0=\frac{1}{2d}q_1$ (for the random walk from $\eta(0)$ not hitting $\partial^{\mathrm{o}}\boldsymbol{\mathrm{R}}(\eta)$, the only possible move for the first step is to $\eta(1)$); 
		
		\item $q_i=\frac{1}{2d}(q_{i-1}+q_{i+1})$ for all $1\le i\le \boldsymbol{\mathrm{L}}(\eta)-1$ (when the aforementioned random walk is at $\eta(i)$, it can only go forward to $\eta(i+1)$ or go backward to $\eta(i-1)$); 
 		
		\item $q_{\boldsymbol{\mathrm{L}}(\eta)}=1$ ($\tau_{\eta(-1)}=0$ a.s. for the random walk from $\eta(-1)$). 
		
	\end{itemize}
By the basic theory for recursive sequences, the only solution is given by 
\begin{equation*}
	q_i= \frac{(d+\sqrt{d^2-1} )^{i+1} -(d-\sqrt{d^2-1} )^{i+1} }{(d+\sqrt{d^2-1} )^{\boldsymbol{\mathrm{L}}(\eta)+1} -(d-\sqrt{d^2-1} )^{\boldsymbol{\mathrm{L}}(\eta)+1} },\ \forall 0\le i\le \boldsymbol{\mathrm{L}}(\eta). 
\end{equation*}
By taking $i=0$, we obtain the formula (\ref{new_4.3}). 
\end{proof}

Now we are ready to establish Proposition \ref{prop_1}. 
\begin{proof}[Proof of Proposition \ref{prop_1}]
	
	We first prove an upper bound for $\mathrm{Es}_{K_n}(\bm{0})$. Arbitrarily take a large enough $M\gg n$. For any $y\in \partial^{\mathrm{o}}B(M)$, by Markov property and Lemma \ref{lemma2.3}, the probability $\mathbb{P}_{\bm{0}}\big(\tau_{K_n\cup \partial^{\mathrm{o}} B(M)}^+=\tau_{y}\big)$ can be written as
	\begin{equation*}
		\begin{split}
			&\frac{1}{2d} \sum_{v_1\in N(\bm{0})}\mathbb{P}_{v_1}\left(\tau_{K_n\cup \partial^{\mathrm{o}} B(M)}=\tau_{y}\right) \\
			 &=\frac{1}{2d} \sum_{v_1\in N(\bm{0})} \sum_{ v_2\in N(x_n^1)\cap \Lambda(n)} G_{K_n} (v_1,v_2) \mathbb{P}_{v_2}\left(\tau^+_{K_n\cup  [\partial^{\mathrm{o}} B(M)]\cup [N(x_n^1)\cap \Lambda(n)]} = \tau_{y} \right).  
		\end{split}		
	\end{equation*}
	Note that $G_{K_n} (v_1,v_2)\le G(\bm{0},\bm{0})$. Moreover, by Obsevation (1) in Remark \ref{remark_observation}, the event $\{\tau^+_{K_n\cup  [\partial^{\mathrm{o}} B(M)]\cup [N(x_n^1)\cap \Lambda(n)]} = \tau_{y}\}$ implies that the random walk must go to $x_n^1$ in the first step, and then move along $\eta_n^{\mathrm{o}}\circ \eta_n^{\mathrm{i}}$. Thus, by Lemma \ref{lemma_move_along}, we have  
 	\begin{equation*}
 		\begin{split}
 				\mathbb{P}_{\bm{0}}\left(\tau^+_{K_n\cup \partial^{\mathrm{o}} B(M)}=\tau_{y}\right) \le 2d\cdot G(\bm{0},\bm{0})\cdot \gamma(\eta_n^{\mathrm{o}}\circ \eta_n^{\mathrm{i}}) 
 				\le C \left(  d+\sqrt{d^2-1}\right)  ^{-\boldsymbol{\mathrm{L}}(\eta_n^{\mathrm{o}})-\boldsymbol{\mathrm{L}}(\eta_n^{\mathrm{i}}) }.
 		\end{split}
 	\end{equation*}
By summing over $y\in \partial^{\mathrm{o}}B(M)$ and letting $M\to \infty$, we obtain
\begin{equation}\label{new_4.4}
	\mathrm{Es}_{K_n}(\bm{0}) \le C\left(  d+\sqrt{d^2-1}\right)  ^{-\boldsymbol{\mathrm{L}}(\eta_n^{\mathrm{o}})-\boldsymbol{\mathrm{L}}(\eta_n^{\mathrm{i}}) }. 
\end{equation}

Next, we prove a lower bound of $\mathrm{Es}_{K_n\setminus \{z\}}(\bm{0})$ for all $z\in K_n\setminus \{\bm{0}\}$. When $z\in K_n^2$, by Observation (4a) in Remark \ref{remark_observation}, $\mathrm{Es}_{K_n\setminus \{z\}}(\bm{0})$ is bounded from below by 
\begin{equation}\label{new_4.5}
	\begin{split}
		& \gamma(\hat{\eta}_{\bm{0}}) \mathbb{P}_{\hat{\eta}_{\bm{0}}(-1)} \left(\tau_{\hat{\eta}^{\mathrm{i}}_{z}(0)}< \tau_{K_n} \right)  \gamma(\hat{\eta}_z^{\mathrm{i}}\circ \hat{\eta}_z^{\mathrm{o}})  \mathbb{P}_{\hat{\eta}_z^{\mathrm{o}}(-1)} \left(\tau_{K_n}=\infty \right)\\
		   & \ge c \left(   d+\sqrt{d^2-1}\right) ^{-0.3n} \mathbb{P}_{\hat{\eta}_{\bm{0}}(-1)} \left(\tau_{\hat{\eta}^{\mathrm{i}}_{z}(0)}< \tau_{K_n} \right)  \mathbb{P}_{\hat{\eta}_z^{\mathrm{o}}(-1)} \left(\tau_{K_n}=\infty \right),
	\end{split}
\end{equation}
where in the last inequality we used Lemma \ref{lemma_move_along} and the fact that (recalling Observation (1) in Remark \ref{remark_observation})
$$
\max\left\lbrace \boldsymbol{\mathrm{L}}(\hat{\eta}_{\bm{0}}),\boldsymbol{\mathrm{L}}(\hat{\eta}_z^{\mathrm{i}}), \boldsymbol{\mathrm{L}}(\hat{\eta}_z^{\mathrm{o}})\right\rbrace \le 0.1n.
$$
Recall that $\min\{\boldsymbol{\mathrm{D}}(\hat{\eta}_{\bm{0}}(-1),K_n),\boldsymbol{\mathrm{D}}(\hat{\eta}^{\mathrm{i}}_{z}(0),K_n) \}\ge \frac{n}{100}$. Moreover, \cite[Proposition 1.5.9]{lawler2013intersections} shows that for a random walk starting from $B_y(\frac{n}{200\sqrt{d}})$, the probability that hits $y$ before $\partial^{\mathrm{i}}B_y(\frac{n}{100\sqrt{d}})$ is of the same order as $n^{2-d}$. Therefore, by strong Markov property and invariance principle, we know that $\mathbb{P}_{\hat{\eta}_{\bm{0}}(-1)} \left(\tau_{\hat{\eta}^{\mathrm{i}}_{z}(0)}< \tau_{K_n} \right)$ is bounded from below by 
\begin{equation*}
	\begin{split}
			 \mathbb{P}_{\hat{\eta}_{\bm{0}}(-1)}\left(\tau_{B_{\hat{\eta}^{\mathrm{i}}_{z}(0)}(\frac{n}{200\sqrt{d}})}<\tau_{K_n} \right) \cdot \min_{w\in \partial^{\mathrm{i}}B_{\hat{\eta}^{\mathrm{i}}_{z}(0)}(\frac{n}{200\sqrt{d}}) }\mathbb{P}_w\left( \tau_{\hat{\eta}^{\mathrm{i}}_{z}(0)}<\tau_{\partial^{\mathrm{i}}B_{\hat{\eta}^{\mathrm{i}}_{z}(0)}(\frac{n}{100\sqrt{d}})} \right) 
			 \ge cn^{2-d}. 
	\end{split}
\end{equation*}
Furthermore, by strong Markov property, invariance principle and Lemma \ref{lemma_escape}, we also know that $\mathbb{P}_{\hat{\eta}_z^{\mathrm{o}}(-1)} \left(\tau_{K_n}=\infty \right)$ is bounded from below by 
\begin{equation*}
	\begin{split}
	\mathbb{P}_{\hat{\eta}_z^{\mathrm{o}}(-1)} \left(\tau_{\partial^{\mathrm{i}}B(4n)}<\tau_{K_n}\right) \cdot \min_{w\in \partial^{\mathrm{i}}B(4n)}\mathbb{P}_{w}\left(\tau_{K_n}=\emptyset \right)\ge c. 
	\end{split}
\end{equation*}
Combining these two estimates and (\ref{new_4.5}), we obtain 
\begin{equation}\label{new_4.6}
	\mathrm{Es}_{K_n\setminus \{z\}}(\bm{0}) \ge cn^{2-d} \left(   d+\sqrt{d^2-1}\right) ^{-0.3n}. 
\end{equation}

When $z\in T_n^{\mathrm{o}}$, by Observation (4b) in Remark \ref{remark_observation}, we have 
\begin{equation*}
	\begin{split}
		\mathrm{Es}_{K_n\setminus \{z\}}(\bm{0}) \ge \gamma(\hat{\eta}_{\bm{0}}) \mathbb{P}_{\hat{\eta}_{\bm{0}}(-1)}  \big(\tau_{\hat{\eta}^{\mathrm{i}}_{x_n^1}(0)}< \tau_{K_n} \big)  \gamma(\hat{\eta}_{x_n^1}^{\mathrm{i}}\circ \bar{\eta}_z^{\mathrm{o}}\circ \hat{\eta}_{z}^{\mathrm{o}})  \mathbb{P}_{\hat{\eta}_z^{\mathrm{o}}(-1)} \left(\tau_{K_n}=\infty \right).  
	\end{split}
\end{equation*}
Thus, with the similar arguments as proving (\ref{new_4.6}), we can get 
\begin{equation}\label{new_4.7}
	\mathrm{Es}_{K_n\setminus \{z\}}(\bm{0}) \ge  cn^{2-d} \left(   d+\sqrt{d^2-1}\right) ^{-0.3n-\boldsymbol{\mathrm{L}}(\eta_n^{\mathrm{o}})}. 
\end{equation}
In the same way, when $z\in T_n^{\mathrm{i}}$, we also have 
\begin{equation}\label{new_4.8}
	\begin{split}
		\mathrm{Es}_{K_n\setminus \{z\}}(\bm{0}) \ge& \gamma(\hat{\eta}_{\bm{0}}) \mathbb{P}_{\hat{\eta}_{\bm{0}}(-1)}  \big(\tau_{\hat{\eta}^{\mathrm{i}}_{z}(0)}< \tau_{K_n} \big)  \gamma(\hat{\eta}_{z}^{\mathrm{i}}\circ \bar{\eta}_z^{\mathrm{i}}\circ \hat{\eta}_{x_n^3}^{\mathrm{o}})  \mathbb{P}_{\hat{\eta}_{x_n^3}^{\mathrm{o}}(-1)} \left(\tau_{K_n}=\infty \right)\\
		\ge &  cn^{2-d} \left(   d+\sqrt{d^2-1}\right) ^{-0.3n-\boldsymbol{\mathrm{L}}(\eta_n^{\mathrm{i}})}. 
	\end{split}
\end{equation}

By (\ref{new_4.6}), (\ref{new_4.7}) and (\ref{new_4.8}), we have: for any $z\in K_n\setminus \{\bm{0}\}$, 
\begin{equation*}
	\mathrm{Es}_{K_n\setminus \{z\}}(\bm{0}) \ge  cn^{2-d}\left(   d+\sqrt{d^2-1}\right) ^{-0.3n-\max\{\boldsymbol{\mathrm{L}}(\eta_n^{\mathrm{o}}),\boldsymbol{\mathrm{L}}(\eta_n^{\mathrm{i}})\}}.
\end{equation*}
Combined with (\ref{new_4.4}), it implies that 
\begin{equation}\label{new_4.9}
	\frac{\mathrm{Es}_{K_n\setminus \{z\}}(\bm{0})}{\mathrm{Es}_{K_n}(\bm{0})} \ge cn^{2-d} \left(   d+\sqrt{d^2-1}\right)^{\min\{\boldsymbol{\mathrm{L}}(\eta_n^{\mathrm{o}}),\boldsymbol{\mathrm{L}}(\eta_n^{\mathrm{i}})\}-0.3n}.
\end{equation}
By (\ref{new_4.9}) and the fact that $\min\{\boldsymbol{\mathrm{L}}(\eta_n^{\mathrm{o}}),\boldsymbol{\mathrm{L}}(\eta_n^{\mathrm{i}})\}\ge n$ (recalling the construction of $K_n$), we conclude the desired bound (\ref{4.1}). 
\end{proof}


\begin{remark}[Generalization of Theorem \ref{thm1} for $\mathbb{Z}^d$ with $d\ge 3$] \label{generalize_3d}
	
	Although our proof of $\psi(\mathbb{Z}^d)=\infty$ for $d\ge 3$ presented in this section relies on the $\mathbb{Z}^d$ structure, its essential component is realizing the geometry of the Klein bottle within the graph. Thus, we expect that this approach can be generalized to a broader family of graphs in high dimensions. 
	
\end{remark}

	\section{Exponential decay of harmonic measure extremum for $\mathbb{Z}^d$ ($d\ge 3$)}\label{section_3d_decay}
	
	In this section, we will present the proof of Theorem \ref{thm2}. However, in the case of $\mathbb{Z}^d$ for $d\ge 3$, it is no longer possible to employ the argument in proving Theorem \ref{coro1} (i.e. remove only one vertex at a time), as the price of removing a single vertex can become arbitrarily large (see Section \ref{section_psi_infinite}). To overcome this obstacle, our strategy is to estimate the price of removing an entire $*$-cluster with respect to its cardinality (see Lemma \ref{lemma_exponential_price}), and then conclude the exponential decay by proving the lower bound for the harmonic measure of a $*$-connected set (see Lemma \ref{lemma_3d_connect_harmonic}).  To be precise, we need the following lemmas. 
	
	\begin{lemma}\label{lemma_exponential_price}
		For any $d\ge 3$, there exists a constant $C_7(d)>0$ such that for any $A\in \mathcal{A}(\mathbb{Z}^d)$ and any $*$-cluster $D\subset A$ with $\bm{0}\notin D$, we have 
		\begin{equation}\label{5.1}
			\frac{\mathbb{H}_{A\setminus D}(\bm{0})}{\mathbb{H}_A(\bm{0})} \le e^{C_7|D|}. 
		\end{equation}
	\end{lemma}

	\begin{lemma}\label{lemma_3d_connect_harmonic}
		For any $d\ge 3$, there exists a constant $C_8(d)>0$ such that for any $*$-connected $A\in \mathcal{A}(\mathbb{Z}^d)$, we have 
		\begin{equation}
			\mathbb{H}_A(\bm{0})\ge e^{-C_8|D|}. 
		\end{equation}
	\end{lemma}

	With the help of these two lemmas, proving Theorem \ref{thm2} is straightforward:
	\begin{proof}[Proof of Theorem \ref{thm2} assuming Lemmas \ref{lemma_exponential_price} and \ref{lemma_3d_connect_harmonic}] 
		Note that a high dimensional version of the tube example presented in Figure \ref{fig:tube} implies that $\mathcal{M}_n(\mathbb{Z}^d)\le e^{-cn}$ for some constant $c(d)>0$.

		For the lower bound in (\ref{1.6}), we enumerate the $*$-clusters of $A$ by $\{D_i\}_{i=0}^{k}$ ($k\in \mathbb{N}^0$), where $D_0$ is the one containing $\bm{0}$. Let $A_l:= A\setminus \cup_{i=1}^{l}D_i$ for $0\le l\le k$ (especially, $A_0=A$ and $A_k=D_0$). Note that $A_l\in \mathcal{A}(\mathbb{Z}^d)$ for all $0\le l\le k$ since $A\in \mathcal{A}(\mathbb{Z}^d)$. Then it is easy to see that 
		\begin{equation*}
			\mathbb{H}_A(\bm{0}) =	\mathbb{H}_{D_0}(\bm{0}) \cdot \prod_{l=1}^{k} \frac{\mathbb{H}_{A_{l-1}}(\bm{0})}{\mathbb{H}_{A_{l}}(\bm{0})}.
		\end{equation*}
		Thus, by Lemmas \ref{lemma_exponential_price} and \ref{lemma_3d_connect_harmonic}, we get 
		\begin{equation*}
			\mathbb{H}_A(\bm{0}) \ge e^{-C_8|D_0|}\cdot e^{-C_7\sum_{l=1}^k|D_l|}. 
		\end{equation*}
		Combined with the fact that $|A|= \sum_{l=0}^k|D_i|$, this concludes the desired lower bound with $C_2= \max\{C_7,C_8\}$. 
	\end{proof}

	We first establish Lemma \ref{lemma_exponential_price} as follows.

	\begin{proof}[Proof of Lemma \ref{lemma_exponential_price}]
		
	Recall that $A_{(\bm{0})}$ is the $*$-cluster of $A$ that contains $\bm{0}$. If $D$ is contained in some finite connected component of $[A_{(\bm{0})}]^c$, then (\ref{5.1}) holds since $\mathbb{H}_{A\setminus D}(\bm{0})=\mathbb{H}_A(\bm{0})$. To avoid this trivial case, we assume that $D\subset [A_{(\bm{0})}]_{\infty}^c$.

	Let $F_1\subset F_2\subset \mathbb{Z}^d$ be subsets with $F_1\cap A=F_2\cap A=D$, which will be determined later. Denote $\widetilde{A}:= A\setminus D$. Recall that for any $j\in \{1,2\}$, $\widehat{F}_j:=[\partial^{\mathrm{i}}_{\infty} (A\cup F_j)]\setminus \widetilde{A}$ and $\widecheck{F}_j:=[\partial^{\mathrm{i}}_{\bm{0}} (A\cup F_j)]\setminus \widetilde{A}$. Also recall that $\mathring{v}_1$ maximizes $G_{\widetilde{A}}(\cdot,v_2)$ in $\widehat{F}_1\cup \widecheck{F}_1$. By adapting the decompositions in Section \ref{subsection_decomposition}, we get the analogues of (\ref{3.1}) and (\ref{3.2}) as follows: 
		\begin{equation*}\label{5.4}
		\begin{split}
			&\mathbb{P}_{\bm{0}}\left( \tau_D<\tau_{\widetilde{A}}=\infty \right)=\sum_{v_1\in \widecheck{F}_1,v_2\in\widehat{F}_2} \mathbb{P}_{\bm{0}}\big(  \tau_{A\cup F_1}=\tau_{v_1}<\infty \big)\\
			&\hskip 2.5cm \cdot \mathbb{P}_{v_2}\left(  \tau_{A\cup F_2}^+=\infty \right)
			 \Big[ \sum_{w\in D} \mathbb{P}_{v_1} \left( \tau_A=\tau_w<\infty \right)   G_{\widetilde{A}}(w,v_2)\Big],
		\end{split}	
	\end{equation*}
		\begin{equation*}\label{5.5}
		\begin{split}
			\mathbb{P}_{\bm{0}}\left( \tau_{F_1}<\tau_{A}^+=\infty\right)
			 =\sum_{v_1\in \widecheck{F}_1,v_2\in \widehat{F}_2}  \mathbb{P}_{\bm{0}}\left( \tau_{A\cup F_1}=\tau_{v_1}<\infty \right) \mathbb{P}_{v_2}\left(  \tau_{A\cup F_2}^+= \infty \right)
			 G_{A}(v_1,v_2). 
		\end{split}
	\end{equation*}
    Based on these two formulas, for the same reason as proving Lemma \ref{lemma_upper_bound}, we have 
	\begin{equation}\label{5.6}
		\frac{\mathrm{Es}_{\widetilde{A}}(\bm{0})}{\mathrm{Es}_{A}(\bm{0})} \le   \max_{v_1\in \widecheck{F}_1,v_2\in \widehat{F}_2} \frac{G_{\widetilde{A}}(v_1,v_2)}{G_{A}(v_1,v_2)}. 
	\end{equation}

	Let $R:= \boldsymbol{\mathrm{D}}(\widetilde{A},D)$. Next, we take $F_1$ and $F_2$ separately in different cases (i.e. when $R< 16d|D|$ and when $R\ge  16d|D|$), and then prove that there exists $C(d)>0$ such that
		\begin{equation}\label{6.4}
			G_{\widetilde{A}}(v_1,v_2) \le e^{C\cdot |D|} G_{A}(v_1,v_2),\ \forall v_1\in \widecheck{F}_1\ \text{and}\ v_2\in \widehat{F}_2.  
		\end{equation}

		\textbf{When $R< 16d|D|$}: In this case, we set $F_1=F_2=D\cup  \partial^{\mathrm{o}}_\infty D$. It follows that $\widehat{F}_j,\widecheck{F}_j\subset \partial^{\mathrm{o}}_\infty D$ for $j\in \{1,2\}$. Hence, $v_1,\mathring{v}_1,v_2\in \partial^{\mathrm{o}}_\infty D$. Since $v_1$ and $v_2$ are connected in $\partial_{\infty,*}^{\mathrm{o}} D$ (by Lemma \ref{lemma_connectivity}), there exists a path $\eta^{1,2}$ from $v_1$ to $v_2$ such that $\boldsymbol{\mathrm{R}}(\eta)\cap A= \emptyset $ and $\boldsymbol{\mathrm{L}}(\eta)\le |\partial_{\infty,*}^{\mathrm{o}} D|\le 3^d|D|$. As a result, we get 
		\begin{equation}\label{new_5.8}
			\mathbb{P}_{v_1}\left(\tau_{v_2}<\tau_{A}\right) \ge  (2d)^{-3^d|D|}. 
		\end{equation}
		By the definition of $R$, there exists a path $\eta'$ from some vertex $v'\in \partial^{\mathrm{o}}_\infty D$ to $\widetilde{A}$ such that $\boldsymbol{\mathrm{L}}(\eta')= R-1$ and $\boldsymbol{\mathrm{R}}(\eta')\cap D=\emptyset$, which implies that $\mathrm{P}_{v'}\left(\tau_{\widetilde{A}}<\tau_D \right)\ge (2d)^{-R+1}$. Moreover, for the same reason as proving (\ref{new_5.8}), we have $\mathbb{P}_{\mathring{v}_1}\left(\tau_{v'}<\tau_{A}\right) \ge  (2d)^{-3^d|D|}$. Thus, by strong Markov property, we get 
		\begin{equation}\label{new_5.9}
			\mathbb{P}_{\mathring{v}_1}\left(\tau_{\widetilde{A}} <\tau_D\right) \ge \mathbb{P}_{\mathring{v}_1}\left(\tau_{v'}<\tau_{A} \right) \mathrm{P}_{v'}\left(\tau_{\widetilde{A}}<\tau_D \right) \ge (2d)^{-3^d|D|-R+1}. 
		\end{equation}
		Combining (\ref{tilde_v1_e2}), (\ref{new_5.8}), (\ref{new_5.9}) and $R< 16d|D|$, we obtain (\ref{6.4}).

	\textbf{When $R\ge 16d|D|$}: We arbitrarily take $x_{\ddagger}\in D$. Since $R\ge 16d|D|$ and $D$ is $*$-connected, one has $D\subset \mathbf{B}_{x_\ddagger}(\tfrac{R}{16})$ and $\mathbf{B}_{x_\ddagger}(\frac{R}{2})\cap \widetilde{A}=\emptyset$. We set $F_1=\mathbf{B}_{x_\ddagger}(\tfrac{R}{8})$ and $F_2=B_{x_\ddagger}(\tfrac{R}{4})$. Note that $v_1,\mathring{v}_1 \in \partial^{\mathrm{i}}\mathbf{B}_{x_\ddagger}(\tfrac{R}{8})$ and $v_2\in \partial^{\mathrm{i}}\mathbf{B}_{x_\ddagger}(\tfrac{R}{4})$. By Lemma \ref{lemma_escape}, 
	\begin{equation}\label{5.8}
		1-\mathbb{P}_{\mathring{v}_1}\left(\tau_D<\tau_{\widetilde{A}} \right) \ge  \mathbb{P}_{\mathring{v}_1}\Big(\tau_{\mathbf{B}_{x_\ddagger}(\tfrac{R}{16})}=\infty \Big)\ge c_5.
	\end{equation}
	Moreover, by Lemma \ref{lemma_compare_green}, we also have 
	\begin{equation}\label{5.9}
	G_A(v_1,v_2) \ge c_3(\tfrac{1}{16},\tfrac{1}{8},\tfrac{1}{4},d)\cdot	G_A(\mathring{v}_1,v_2).  
	\end{equation}
	Combining (\ref{tilde_v1_e1}), (\ref{5.8}) and (\ref{5.9}), we conclude (\ref{6.4}). By (\ref{5.6}) and (\ref{6.4}), we get
	\begin{equation}\label{5.10}
		\frac{\mathrm{Es}_{\widetilde{A}}(\bm{0})}{\mathrm{Es}_A(\bm{0})}\le e^{C\cdot |D|} .
	\end{equation}

	By Lemma \ref{lemma_2.5}, one has $\mathrm{cap}(A)\le \mathrm{cap}(\{\bm{0}\})\cdot |D|+\mathrm{cap}(\widetilde{A})$ and $\mathrm{cap}(\widetilde{A})\ge \mathrm{cap}(\{\bm{0}\})$, which implies that $\frac{\mathrm{cap}(A)}{\mathrm{cap}(\widetilde{A})}\le 1+ |D|$. Combined with (\ref{5.10}) and $\frac{\mathbb{H}_{\widetilde{A}}(\bm{0})}{\mathbb{H}_A(\bm{0})}=\frac{\mathrm{Es}_{\widetilde{A}}(\bm{0})}{\mathrm{Es}_A(\bm{0})}\cdot \frac{\mathrm{cap}(A)}{\mathrm{cap}(\widetilde{A})}$ (by (\ref{2.14})), we complete the proof of Lemma \ref{lemma_exponential_price}. 
	\end{proof}

	Subsequently, we demonstrate Lemma \ref{lemma_3d_connect_harmonic}. 
	
	\begin{proof}[Proof of Lemma \ref{lemma_3d_connect_harmonic}]
		 
		 Since $A$ is $*$-connected and contains $\bm{0}$, we know that $A\subset \mathbf{B}(d|A|)$. With the same argument as proving (\ref{new_5.9}), we have 
		 \begin{equation}\label{new_5.13}
		 	\mathbb{P}_{\bm{0}}\left(\tau_{\partial^{\mathrm{i}}\mathbf{B}(2d|A|)}<\tau_{A}   \right) \ge e^{-C|A|}.
			\end{equation}
		 Moreover, by Lemma \ref{lemma_escape}, we have: for any $w\in \partial^{\mathrm{i}}\mathbf{B}(2d|A|)$, 
		 \begin{equation}\label{new_5.14}
		 	\mathbb{P}_w\left(\tau_{A} = \infty \right) \ge c_5.
		 \end{equation}
		 By strong Markov property, (\ref{new_5.13}) and (\ref{new_5.14}), we obtain 
		 \begin{equation}\label{new_5.15}
		 	\mathrm{Es}_A(\bm{0})\ge  \mathbb{P}_{\bm{0}}\left(\tau_{\partial^{\mathrm{i}}B(2d|A|)}<\tau_{A}   \right)\cdot \min_{w\in \partial^{\mathrm{i}}B(2d|A|)} \mathbb{P}_w\left(\tau_{A} = \infty \right) \ge c_5 e^{-C|A|}.
		 \end{equation}
		 By Lemma \ref{lemma_2.5}, one has $\mathrm{cap}(A)\le \mathrm{cap}(\{\bm{0}\})|A|$. Combined with (\ref{new_5.15}), it completes the proof of Lemma \ref{lemma_3d_connect_harmonic}. 
	\end{proof}


		\section*{Acknowledgments}
		
			The collaboration on this paper has much to thank the workshop that took place at the Technion, funded by DFG - 508398298. The authors would like to thank Noam Berger, Yifan Gao and Xinyi Li for fruitful discussions. YZ was supported by NSFC - 11901012  and National Key R\&D Program of China -
2020YFA0712902.

		\bibliographystyle{plain}
		\bibliography{ref}
		
	\end{document}